\documentclass[11pt]{article}
\usepackage[margin=1in]{geometry}

\usepackage{amsmath}
\DeclareMathAlphabet{\mathcal}{OMS}{cmsy}{m}{n}
\usepackage{amssymb}
\usepackage{amsthm}
\usepackage{amsfonts}
\usepackage{mathtools}
\usepackage{mymacros}
\usepackage{braket}
\usepackage{bbm}
\usepackage{graphicx}
\usepackage{caption}
\usepackage{subcaption}
\usepackage{listings}
\usepackage[ruled,vlined,linesnumbered]{algorithm2e}
\usepackage{xcolor}
\usepackage{multirow}
\usepackage{lscape}
\usepackage{pgf,interval}
\usepackage{tabularx}
\usepackage{booktabs}
\usepackage[inline]{enumitem}
\usepackage[hidelinks]{hyperref}%
\usepackage{pgfplots}
\usepackage[title]{appendix}

\title{
    {\bf Bilevel Mixed-Integer Linear Program\\with Binary Tender} 
    \author{
        \normalsize {\bf Bo Zhou}, {\bf Ruiwei Jiang}, and {\bf Siqian Shen} \\
        \small Department of Industrial and Operations Engineering\\
        \small University of Michigan, Ann Arbor, MI 48109\\
        \small Email: \{bozum, ruiwei, siqian\}@umich.edu\\
    }
}
\date{}

\begin{document}

    \maketitle

    \begin{abstract}
        \noindent
        Bilevel programs model sequential decision interactions between two sets of players and find wide applications in real-world complex systems.
        In this paper, we consider a bilevel mixed-integer linear program with binary tender, wherein the upper and lower levels are linked via binary decision variables and both levels may involve additional mixed-integer decisions.
        We recast this bilevel program as a single-level formulation through a value function for the lower-level problem and then propose valid inequalities to replace and iteratively approximate the value function.
        We first derive a family of Lagrangian-based valid inequalities that give a complete description of the value function, providing a baseline method to obtain exact solutions for the considered class of bilevel programs.
        To enhance the strength of this approach, we further investigate another two types of valid inequalities.
        First, when the lower-level value function has intrinsic special properties such as supermodularity or submodularity, we exploit such properties to separate the Lagrangian-based inequalities quickly.
        Second, we derive decision rule-based valid inequalities, where linear decision rules and learning techniques are explored respectively.
        We demonstrate the effectiveness and efficiency of the proposed methods in extensive numerical experiments, including instances of general bilevel mixed-integer programs and those of a facility location interdiction problem.

        ~

    \end{abstract}

    \section{Introduction}\label{sec:intro}
        In a generic bilevel program (BP), a leader and a follower solve their own decision-making problems in an interactive way:
        the leader's decision made in the upper level will affect the follower's problem solved at the lower level (e.g., the leader's decision is involved in the objective function and/or constraints of the follower's problem) and vice versa.
        We denote by vectors $x$ and $y$ the leader's and the follower's decisions, respectively, $n_{x}$ and $n_{y}$ as their dimensions, respectively, and $m_{u}$ and $m_{\ell}$ as the number of upper-level and lower-level constraints, respectively.
        Then, a bilevel mixed-integer linear program (MILP) can be described formally as
        \begin{subequations}\label{eq:bilevel}
            \begin{align}
                \min_{x\in X,y\in Y}~&c_{u}^{\top}x+d_{u}^{\top}y\\
                \textrm{s.t.}~&A_{u}x+B_{u}y\leq h_{u}\label{eq:bilevel2}\\
                &y\in\arg\max_{y'\in Y}~d_{\ell}^{\top}y'\label{eq:bilevel3}\\
                &~~~~~~~~~~~~\textrm{s.t.}~A_{\ell}x+B_{\ell}y'\leq h_{\ell}, \label{eq:bilevel4}
            \end{align}
        \end{subequations}
        where $x\in X\subseteq\mathbb{R}^{n_{x}}$ and $y,y'\in Y\subseteq\mathbb{R}^{n_{y}}$;
        $c_{u}\in\mathbb{R}^{n_{x}}$, $d_{u}\in\mathbb{R}^{n_{y}}$, $A_{u}\in\mathbb{R}^{m_{u}\times n_{x}}$, $B_{u}\in\mathbb{R}^{m_{u}\times n_{y}}$, and $h_{u}\in\mathbb{R}^{m_{u}}$ are upper-level coefficients;
        and $d_{\ell}\in\mathbb{R}^{n_{y}}$, $A_{\ell}\in\mathbb{R}^{m_{\ell}\times n_{x}}$, $B_{\ell}\in\mathbb{R}^{m_{\ell}\times n_{y}}$, and $h_{\ell}\in\mathbb{R}^{m_{\ell}}$ are lower-level coefficients.
        For the lower-level problem, without loss of generality, we assume that $x$ only appears in the feasible region, because the lower-level objective function can be relegated to the constraints through its hypographical form.
        Meanwhile, the leader's objective function depends on both the leader's decision $x$ and the follower's \emph{optimal} decision $y$, which is in turn a function of $x$ defined through~\eqref{eq:bilevel3}--\eqref{eq:bilevel4}.
        Additionally, the upper-level feasible region may also depend on $y$ as in constraint~\eqref{eq:bilevel2}.
        We note that there may exist multiple $y$'s that are (equally) optimal to the lower-level problem and we consider the \emph{optimistic} case in this paper, which means that the follower adopts the $y$ best complying with the leader's objective.
        There have also been related studies on pessimistic cases of BPs with continuous variables and linear constraints, and we refer the interested readers to \cite{wiesemann2013pessimistic,liu2018pessimistic} for more details.

        A general approach to solving the BP in model~\eqref{eq:bilevel} is the value function reformulation (VFR), which incorporates the follower's decision and lower-level feasible region into the upper-level problem, giving rise to an equivalent, single-level reformulation
        \begin{subequations}\label{eq:VFR}
            \begin{align}
                \min_{x\in X,y\in Y}&~c_{u}^{\top}x+d_{u}^{\top}y\label{eq:upper objective}\\
                \textrm{s.t.}&~A_{u}x+B_{u}y\leq h_{u}\\
                &~A_{\ell}x+B_{\ell}y\leq h_{\ell}\label{eq:lower feasibility}\\
                &~d_{\ell}^{\top}y\geq\phi(x),\label{eq:optimality}
            \end{align}
        \end{subequations}
        where the value function $\phi:X\rightarrow\mathbb{R}$ is defined as
        \begin{equation}
            \begin{aligned}\label{eq:value function}
                \phi(x):=\max_{y'\in Y}&~d_{\ell}^{\top}y'\\
                \textrm{s.t.}&~B_{\ell}y'\leq h_{\ell}-A_{\ell}x,
            \end{aligned}
        \end{equation}
        representing the optimal objective value of the lower-level problem as a function of $x$.
        In the VFR~\eqref{eq:VFR}, constraint \eqref{eq:optimality} designates the optimality of $y$ for the lower-level problem and is also known as bilevel feasibility.
        By relaxing \eqref{eq:optimality} from \eqref{eq:VFR}, we obtain the so-called high-point relaxation (HPR) \cite{kleinert2021survey}, which provides a lower bound for the original BP~\eqref{eq:bilevel}.
        In this paper, we focus on bilevel MILPs with binary tender, which we next formalize.
        \begin{assumption}\label{asp:binary}
            The upper-level decision variables that appear in the formulation at the lower level are binary-valued.
        \end{assumption}
        Assumption \ref{asp:binary} arises in various applications of BP in the real world, including energy system expansion planning~\cite{KabirifarFotuhi-Firuzabad-3072}, charging station planning~\cite{li2022public}, competitive facility location~\cite{qi2021sequential}, and network interdiction~\cite{smith2020survey}.
        This assumption is mild because an integer tender variable can be exactly represented by (logarithmically many) binary variables and a continuous tender variable can be approximated to arbitrary accuracy by binary variables.
        Note that the problems at both levels may involve general decision variables (i.e., continuous and/or integer), and we only assume that the entries of $x$ appearing in the lower-level formulation are binary-valued.
        For ease of exposition, we shall assume that $x\in X=\{0,1\}^{n_{x}}$ in the remainder of the paper, but it is clear that the very same approach that we develop later works when $x$ is mixed-integer but satisfies Assumption \ref{asp:binary}.

        \subsection{Literature Review}
            BPs are appealing for modeling problems that involve sequential decision interactions from two or multiple players.
            Their hierarchical decision processes arise in a wide range of real-world applications, including energy \cite{ZhouFang-4358}, security \cite{moug2025stochastic}, transportation \cite{song2016risk, lei2018stochastic}, market design~\cite{nasiri2020bi}, and machine learning \cite{pmlr-v139-wang21ad, hong2020two, bao2021stability}.
            However, it has been proved that even the simplest BP, of which both levels are linear programs, is NP-hard and computationally intractable \cite{jeroslow1985polynomial,hansen1992new}.
            Existing algorithms for BPs depend on the properties of the lower-level problem, as well as the linking variables between the two levels~\cite{colson2007overview,sinha2017review}.
            In case a BP is continuous (i.e., both $x$ and $y$ consist of continuous decision variables only), descent methods based on explicit gradients or implicit gradients become applicable and can find a local optimal solution to the BP \cite{liu2021towards, jiang2023conditional}.
            Yet, descent methods do not guarantee global optimality, and additional assumptions are usually required for their convergence, such as lower-level convexity and singleton (i.e., given $x$, there is a unique $y$ optimal to the lower-level problem).
            In case the lower-level problem is continuous and convex, constraint~\eqref{eq:optimality} can be explicitly replaced by complementarity constraints~\cite{Fortuny81} or bilinear constraints~\cite{zare2019note,mccormick1976computability} through the KKT conditions of the lower-level problem or strong duality, respectively.
            Consequently, solving BPs boils down to solving the ensuing single-level but nonconvex and nonlinear reformulation \cite{colson2005bilevel,Bard98}.
            However, such luxury is immediately lost if the lower-level problem involves discrete decision variables (e.g., $y$ is binary or mixed-binary) or nonconvexity, where neither KKT conditions nor strong duality approaches may be able to capture the (parametric) global optimality at the lower level.
            In that case, the closed-form expression of $\phi(x)$ is either non-existent or highly intractable, prohibiting solving the BP effectively.
            Therefore, more effort is required for more general BPs with discrete decision variables or nonconvex objectives/constraints at the lower level.

            To address this issue, integer programming techniques \cite{kleinert2021survey, smith2020survey, beck2022survey} are usually employed and can be categorized into three streams using (i) value function representation, (ii) cutting planes, or (iii) branch-and-bound algorithm.
            The first stream (i) aims at obtaining a closed-form expression of the lower-level value function $\phi(x)$, through methods such as the multi-parametric representation theory \cite{avraamidou2019multi}, the properties of bilevel integer program value functions \cite{zhang2021bilevel}, iterative approximation \cite{dempe2017solving}, or neural network approximation \cite{zhoulearning}.
            Yet, such methods are limited to their special problem structures or can only provide bound information.
            The second stream (ii) constructs cutting planes to iteratively approximate the optimality condition \eqref{eq:optimality}.
            Most algorithms of this stream are based on the projection of the HPR feasible region on $y$ and need to introduce indicator constraints \cite{mitsos2010global, lozano2017value}.
            For example, \cite{zeng2014solving} proposes to adopt column-and-constraint generation to generate proper $y$ and the corresponding cutting planes.
            Building on this idea, \cite{yue2019projection, merkert2022exact} focus on how to properly handle the indicator constraint under special problem structures.
            Different from the projection approach, \cite{caramia2015enhanced} proposes a valid inequality to cut off bilevel infeasible points (pairs of $x$ and $y$ that violate constraint \eqref{eq:optimality}), which yet introduces an additional lower-level problem for each bilevel infeasible point identified.
            The third stream (iii) stems from \cite{bard1992algorithm}, which adapts the classic branch-and-bound algorithm to BPs with discrete lower-level variables.
            Along this stream, different types of valid inequalities have been developed to improve the computational efficiency, including the integer no-good cut \cite{denegre2011interdiction}, multi-disjunction cut \cite{wang2017watermelon}, intersection cut \cite{fischetti2018use, fischetti2017new}, and disjunctive cut \cite{gaar2023socp, horlander2024using}.
            In particular, \cite{tahernejad2020branch} develops the \texttt{MibS} solver for general bilevel MILPs, which incorporates most of the above cuts and some cuts for specially structured problems.
            Nevertheless, most of the above methods are dedicated to BPs with general mixed-integer or integer linking variables.
            To the best of our knowledge, although a few methods (e.g., \cite{dempe2017solving}) restrict to binary tender, they do not fully exploit this information.
            In this paper, we will take advantage of the ``binary tender'' assumption and show that this enables novel valid inequalities for solving BPs more efficiently.

        \subsection{Contributions and Paper Organization}
            We develop exact algorithms to solve bilevel MILPs with binary tender, in which novel valid inequalities are developed for higher computational efficiency.
            Our main contributions include:
            \begin{enumerate}
                \item We derive a family of Lagrangian-based valid inequalities that give a complete description of the optimality condition \eqref{eq:optimality}, providing a baseline method to obtain exact solutions under general conditions.
                    We provide both exact and quick calculation of Lagrangian coefficients.
                \item For special cases that the value function $\phi(x)$ has intrinsic special properties such as supermodularity or submodularity, we exploit such properties to efficiently calculate the exact Lagrangian coefficients or produce stronger valid inequalities.
                    We further extend to quasi-supermodular or quasi-submodular cases, wherein these properties hold only after fixing the discrete decision variables at the lower level.
                \item We propose additional decision rule-based valid inequalities, which arise from solving the lower-level problem approximately using decision rules.
                    We explore two different decision rules: i) iteratively updated linear decision rules and ii) trained nonlinear decision rules from past solves of the lower-level problem.
                \item We conduct extensive numerical experiments using test instances for general bilevel MILPs and a facility location interdiction problem to demonstrate the effectiveness and performance of our proposed methods.
            \end{enumerate}

            The remainder of the paper is organized as follows.
            Sections \ref{sec:Lagrangian}--\ref{sec:DR} propose the Lagrangian-based, special property-based, and decision rule-based valid inequalities, respectively.
            Section \ref{sec:results} conducts numerical experiments and Section \ref{sec:conclusions} draws conclusions.

            \emph{Notations:}
            For integers $n$, we define $[n]:=\{1,2,\ldots,n\}$.
            For $a\in\mathbb{R}$, we define $a^{+}:=\max\{a,0\}$ and $a^{-}:=\min\{a,0\}$.
            For $a', a'' \in\mathbb{R}^{n}$, we define $a' \vee a'' = [\max\{ a'_1, a''_1\}, \ldots, \max\{a'_n, a''_n\}]^{\top}$ and $a' \wedge a'' = [\min\{a'_1, a''_1\}, \ldots, \min\{a'_n, a''_n\}]^{\top}$.
            For $a\in\mathbb{R}^{n}$ and $i\in[n]$, $a_{i}$ means the $i$th entry of $a$ and $a_{-i}\in\mathbb{R}^{n-1}$ denotes the vector obtained by removing $a_{i}$ from $a$.
            For $i\in\mathbb{Z}_{+}$, $e_{i}$ denotes a vector with suitable dimension, with entry $i$ being 1 and all other entries being 0.

    \section{Lagrangian-based Valid Inequalities}\label{sec:Lagrangian}
        The main difficulty in exactly solving the reformulation \eqref{eq:VFR} lies in constraint \eqref{eq:optimality}.
        In this section, we develop a family of Lagrangian-based valid inequalities to represent \eqref{eq:optimality}.
        The main idea is to rewrite the value function \eqref{eq:value function} as
        \begin{equation}\label{eq:copy}
            \begin{aligned}
                \phi(x)=\max_{z\in[0,1]^{n_{x}}}&~\psi(z)\\
                \text{s.t.}&~z=x,
            \end{aligned}
        \end{equation}
        where $\psi(z)$ with $z\in[0,1]^{n_{x}}$ is an extension of $\phi(x)$ with $x\in\{0,1\}^{n_{x}}$, that is, \(\psi(z) = \phi(z)\) whenever \(z \in\{0,1\}^{n_{x}}\). This extension enables Lagrangian relaxation and the subsequent construction of three types of valid inequalities.

        There are multiple ways to extend from $\phi(x)$ with $x\in\{0,1\}^{n_{x}}$ to $\psi(z)$ with $z\in[0,1]^{n_{x}}$.
        For the ease of calculating Lagrangian coefficients, we consider extensions with such structure that, with other entries fixed, $\psi(z)$ is linear in each individual $z_{i}$ for each $i\in[n_{x}]$.
        For example, we can set
        \begin{align*}
            \psi(z):=\sum_{x\in\{0,1\}^{n_{x}}}\phi(x)\prod_{i\in[n_{x}]:x_{i}=1}z_{i}\prod_{i\in[n_{x}]:x_{i}=0}(1-z_{i}).
        \end{align*}

         \subsection{Formulations and Valid Inequalities}
                We define the validity and tightness of inequalities as follows.
                \begin{definition}\label{def:tight}
                    Consider two inequalities $f(x)\leq0$ and $g(x)\leq0$ with respect to decision variables \(x \in X\). Then,\\
                    (i) we say $f(x)\leq0$ is valid for inequality $g(x)\leq0$ if $f(x)\leq g(x)$ for all $x\in X$;\\
                    (ii) we say $f(x)\leq0$ is valid and tight for inequality $g(x)\leq0$ if $f(x)\leq g(x)$ for all $x\in X$ and there exists an $\hat{x}\in X$ such that $g(\hat{x})=f(\hat{x})$. In particular, we say \(f(x) \leq 0\) is tight at \(\hat{x}\).
                \end{definition}
                In the following, we exploit Lagrangian duality to derive valid and tight inequalities for \eqref{eq:optimality}.

            \subsubsection{Penalty-based Valid Inequality}
                We first consider the penalty method and design a Lagrangian function by relaxing $z=x$ as
                \begin{equation}\label{eq:penalty}
                    L_{p}(x, \rho)=\max_{z\in[0,1]^{n_{x}}}~\psi(z)-\rho\left(1^{\top}x+1^{\top}z-2x^{\top}z\right),
                \end{equation}
                where $\rho\in\mathbb{R}_{+}$ is a penalty coefficient.
                For any $x\in\{0,1\}^{n_{x}}$, $z=x$ is feasible to the right-hand side of \eqref{eq:penalty} and hence $\phi(x)\leq L_{p}(x,\rho)$, implying the weak duality $\phi(x)\leq\min_{\rho\in\mathbb{R}_{+}}L_{p}(x,\rho)$.
                In fact, because $x$ is binary-valued, Lemma~\ref{lem:penalty} below shows that the strong duality holds. We note that a similar result has been shown in Theorem 3 of \cite{zou2019stochastic}.
                Before presenting this result, we define the projection of the feasible region defined by the lower-level constraints onto $x$ as
                \begin{equation}\label{eq:lower-level feasibility}
                    X_{LF}:=\left\{x\in\{0,1\}^{n_{x}}\left|~\exists y'\in Y, B_{\ell}y'\leq h_{\ell}-A_{\ell}x\right.\right\}.
                \end{equation}
                \begin{lemma}
                \label{lem:penalty}
                    For any $x\in\{0,1\}^{n_{x}}$, we have
                    \begin{equation}\label{eq:penalty-strong}
                        \phi(x)=\min_{\rho\in\mathbb{R}_{+}}L_{p}(x,\rho).
                    \end{equation}
                    In addition, if $x\in X_{LF}$, then there exists a $\rho^{*}(x)\in\mathbb{R}_{+}$ such that for all $\rho\geq\rho^{*}(x)$, we have $\phi(x)=L_{p}(x,\rho)$.
                \end{lemma}
                \begin{proof}\color{black}
                    We consider two cases according to the feasibility of $x$.

                    Case 1: If $x\notin X_{LF}$,
                        we have $\phi(x)=-\infty$ according to \eqref{eq:value function}, and hence, $L_{p}(x, \rho)\geq-\infty$ due to the weak duality.
                        We further consider two cases.\\
                        Case 1.1: If $L_{p}(x, \rho)=-\infty$,
                            which means $\psi(z)$ is infeasible for all $z\in[0,1]^{n_{x}}$, we have
                            \begin{equation*}
                                \min_{\rho\in\mathbb{R}_{+}}L_{p}(x,\rho)=-\infty=\phi(x).
                            \end{equation*}
                        Case 1.2: If $L_{p}(x, \rho)>-\infty$,
                            supposing that $z^*$ is the optimal solution to the right-hand solution of \eqref{eq:penalty}, we have
                            \begin{equation*}
                                L_{p}(x,\rho)=\psi(z^*)-\rho\left(1^{\top}x+1^{\top}z^*-2x^{\top}z^*\right).
                            \end{equation*}
                            Considering that $x$ is infeasible to $\phi$, $x$ is also infeasible to $\psi$.
                            While $z^*$ is feasible to $\psi$, $z^*\neq x$ must hold, which implies that
                            \begin{equation*}
                                1^{\top}x+1^{\top}z^*-2x^{\top}z^*\geq\|x-z^*\|^2_2>0.
                            \end{equation*}
                            With the bounded $\psi(z^*)$, when $\rho\rightarrow+\infty$, we have $L_{p}(x,\rho)\rightarrow-\infty$.
                            Hence, we have
                            \begin{equation*}
                                \min_{\rho\in\mathbb{R}_{+}}L_{p}(x,\rho)=-\infty=\phi(x).
                            \end{equation*}

                    Case 2: If $x\in X_{LF}$,
                        we have $\phi(x)>-\infty$ according to \eqref{eq:value function}, and hence, $L_{p}(x, \rho)>-\infty$ due to the weak duality.
                        Supposing that $z^*(x,\rho)$ is the optimal solution to the right-hand solution of \eqref{eq:penalty}, we have
                        \begin{equation*}
                            L_{p}(x,\rho)=\psi(z^*(x,\rho))-\rho\left(1^{\top}x+1^{\top}z^*(x,\rho)-2x^{\top}z^*(x,\rho)\right),
                        \end{equation*}
                        where
                        \begin{equation*}
                            1^{\top}x+1^{\top}z^*(x,\rho)-2x^{\top}z^*(x,\rho)\geq\|x-z^*(x,\rho)\|^2_2\geq0.
                        \end{equation*}
                        First, we prove that there exists $\rho^{*}\in\mathbb{R}_{+}$ such that for all $\rho\geq\rho^{*}$, $1^{\top}x+1^{\top}z^{*}(x,\rho)-2x^{\top}z^{*}(x,\rho)=0$.
                        Suppose the contrary that for all $\rho^{*}\in\mathbb{R}_{+}$, there exists $\rho\geq\rho^{*}$ such that $1^{\top}x+1^{\top}z^{*}(x,\rho)-2x^{\top}z^{*}(x,\rho)>0$.
                        When $\rho^{*}\rightarrow+\infty$, we have $\rho=+\infty$ and then $L_{p}(x, \rho)=-\infty$, which contradicts with $L_{p}(x, \rho)>-\infty$.
                        Hence, we finish this part of the proof.
                        Considering that the given $x$ may influence the value of $\rho^{*}$, we use $\rho^{*}(x)$ in the following for clarity.\\
                        Second, because for all $\rho\geq\rho^{*}(x)$, $1^{\top}x+1^{\top}z^{*}(x,\rho)-2x^{\top}z^{*}(x,\rho)=0$ holds, we have $z^{*}(x,\rho)=x$, which implies $L_{p}(x,\rho)=\psi(x)=\phi(x)$.
                        Furthermore, we have
                        \begin{equation*}
                            \min_{\rho\in\mathbb{R}_{+}}L_{p}(x,\rho)\leq\min_{\rho\geq\rho^{*}(x)}L_{p}(x,\rho)=\phi(x).
                        \end{equation*}
                        Combining the weak duality that $\phi(x)\leq\min_{\rho\in\mathbb{R}_{+}}L_{p}(x,\rho)$, we have $\phi(x)=\min_{\rho\in\mathbb{R}_{+}}L_{p}(x,\rho)$.
                        This completes the overall proof.
                \end{proof}

                According to Lemma \ref{lem:penalty}, whenever $x$ is feasible to $\phi$ and $\rho$ is sufficiently large, $L_{p}(x,\rho)=\phi(x)$ holds.
                Then, we can derive the following corollary.
                \begin{corollary}\label{crl:penalty}
                    There exists a constant $\hat{\rho}\in\mathbb{R}_{+}$ such that for any $\rho\geq\hat{\rho}$ and $x\in\{0,1\}^{n_{x}}\cap X_{LF}$, we have
                    \begin{equation}\label{eq:penalty-strong2}
                        \phi(x)=L_{p}(x,\rho).
                    \end{equation}
                \end{corollary}
                \begin{proof}\color{black}
                    According to Lemma \ref{lem:penalty}, for any $x\in\{0,1\}^{n_{x}}\cap X_{LF}$, there exists $\rho^{*}(x)\in\mathbb{R}_{+}$ such that for all $\rho\geq\rho^{*}(x)$, we have $\phi(x)=L_{p}(x,\rho)$.
                    We define $\hat{\rho}:=\max_{x\in\{0,1\}^{n_{x}}\cap X_{LF}}\rho^{*}(x)$.
                    Then, for any $x\in\{0,1\}^{n_{x}}\cap X_{LF}$ and any $\rho\geq\hat{\rho}$, we have $\rho\geq\rho^{*}(x)$ and thus $\phi(x)=L_{p}(x,\rho)$.
                \end{proof}

                Based on Corollary \ref{crl:penalty}, we derive a penalty-based valid inequality in Proposition \ref{pps:penalty-cut}.
                \begin{proposition}\label{pps:penalty-cut}
                    For any $z\in[0,1]^{n_{x}}$, the following inequality is valid for any $(x,y) \in X \times Y$ satisfying \eqref{eq:optimality}:
                    \begin{equation}\label{eq:penalty-cut}
                        d_{\ell}^{\top}y\geq \psi(z)-\hat{\rho}\left(1^{\top}x+1^{\top}z-2x^{\top}z\right),
                    \end{equation}
                    where $\hat{\rho}$ is a sufficiently large constant.
                    In addition, if $z\in\{0,1\}^{n_{x}}$, then \eqref{eq:penalty-cut} is valid and tight for \eqref{eq:optimality} and can be rewritten as
                    \begin{equation}\label{eq:penalty-cut2}
                        d_{\ell}^{\top}y\geq \phi(z)-\hat{\rho}\left(1^{\top}x+1^{\top}z-2x^{\top}z\right).
                    \end{equation}
                \end{proposition}
                \begin{proof}\color{black}
                    According to Corollary \ref{crl:penalty}, for any $x\in\{0,1\}^{n_{x}}\cap X_{LF}$, we have $\phi(x)=L_{p}(x,\hat{\rho})$.
                    In the VFR \eqref{eq:VFR}, the lower-level feasibility $X_{LF}$ has been satisfied by \eqref{eq:lower feasibility}, and therefore, we can directly replace $\phi(x)$ in \eqref{eq:optimality} by $L_{p}(x,\hat{\rho})$, leading to
                    \begin{equation*}
                        d_{\ell}^{\top}y\geq L_{p}(x,\hat{\rho})=\max_{z\in[0,1]^{n_{x}}}\left\{\psi(z)-\hat{\rho}\left(1^{\top}x+1^{\top}z-2x^{\top}z\right)\right\},
                    \end{equation*}
                    which implies
                    \begin{equation*}
                        d_{\ell}^{\top}y\geq\psi(z)-\hat{\rho}\left(1^{\top}x+1^{\top}z-2x^{\top}z\right),~\forall z\in[0,1]^{n_{x}}.
                    \end{equation*}
                    Hence, for any $z\in[0,1]^{n_{x}}$, \eqref{eq:penalty-cut} is valid for \eqref{eq:optimality}.

                    For any $z\in\{0,1\}^{n_{x}}\subset[0,1]^{n_{x}}$, we have $\psi(z)=\phi(z)$, and hence, \eqref{eq:penalty-cut} can be rewritten as
                    \begin{equation*}
                        \begin{aligned}
                            d_{\ell}^{\top}y\geq\phi(z)-\hat{\rho}\left(1^{\top}x+1^{\top}z-2x^{\top}z\right).
                        \end{aligned}
                    \end{equation*}
                    By fixing $x$ on the right-hand side as $z$, we obtain
                    \begin{equation*}
                        d_{\ell}^{\top}y\geq \phi(z),
                    \end{equation*}
                    which implies that \eqref{eq:penalty-cut2} is valid and tight for \eqref{eq:optimality}.
                    This completes the proof.
                \end{proof}

                \begin{example}\label{exp:penalty}
                    We illustrate the penalty-based valid inequality \eqref{eq:penalty-cut2}.
                    First, note that we can reformulate \eqref{eq:penalty-cut2} as
                    \begin{equation*}
                        d_{\ell}^{\top}y\geq \phi(z)-\rho\left\|x-z\right\|_{1}
                    \end{equation*}
                    because $x,z\in\{0,1\}^{n_{x}}$.
                    When $n_{x}=1$, the inequality reduces to $d_{\ell}^{\top}y\geq \phi(z)-\rho|x-z|$.
                    Figure \ref{fig:example-penalty} illustrates the valid inequality with $z=1$.
                    The shaded part depicts the epigraph defined by the right-hand side of \eqref{eq:penalty-cut2} with respect to a general $\rho \geq \hat{\rho}$.
                    Note that the inequality \eqref{eq:penalty-cut2} is tight at $x=1$.
                    When $\rho\rightarrow+\infty$, the angle of the cut tends to zero.
                    This intuitively explains the existence of $\rho^*(x)$ in Lemma \ref{lem:penalty} and that of $\hat{\rho}$ in Corollary \ref{crl:penalty} because there are finitely many $x$-values.
                    In contrast, as $\rho$ decreases, the angle of the cut grows and eventually touches the point $(0,\phi(0))$.
                    This pertains to the threshold $\hat{\rho}$, as depicted by the red solid line.
                    Intuitively, the cut remains valid whenever $\rho \geq \hat{\rho}$.
                    Therefore, $\rho=\hat{\rho}$ provides the strongest penalty-based valid inequality.
                    \begin{figure}[!htbp]
                        \begin{center}
                            \vspace{-0ex}
                            \includegraphics[width=0.5\columnwidth]{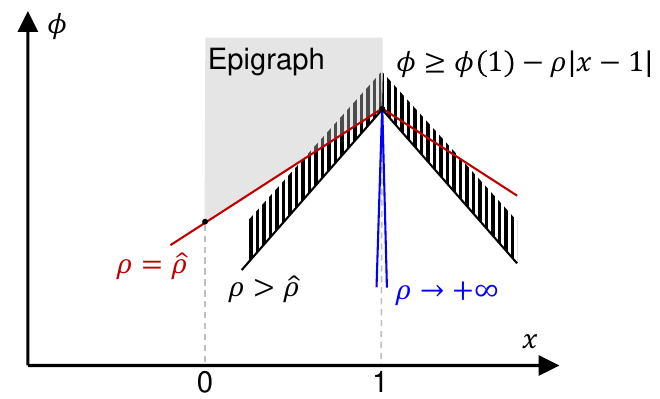}
                        \end{center}
                        \vspace{-4ex}
                        \caption{An illustrative example of the penalty-based valid equality when $n_{x}=1$.}
                        \vspace{-2ex}
                        \label{fig:example-penalty}
                    \end{figure}
                \end{example}

            \subsubsection{Lagrangian-based Valid Inequality}
                Next, we explore the Lagrangian method and design a Lagrangian function by relaxing $z=x$ as
                \begin{equation}\label{eq:LR}
                    L_{\ell}(x, \lambda)=\max_{z\in[0,1]^{n_{x}}}~\psi(z)-\lambda^{\top}(x-z),
                \end{equation}
                where $\lambda\in\mathbb{R}^{n_{x}}$ is a Lagrangian multiplier.
                Obviously, for any $x\in\{0,1\}^{n_{x}}$, $z=x$ is feasible to the right-hand side of \eqref{eq:LR}.
                Hence, $\phi(x)\leq L_{\ell}(x,\lambda)$ always holds, which implies the weak duality $\phi(x)\leq\min_{\lambda\in\mathbb{R}^{n_{x}}}L_{\ell}(x,\lambda)$.
                Furthermore, because $x$ is binary-valued, the Lagrangian dual admits the strong duality, as summarized in Lemma \ref{lem:LR}.
                \begin{lemma}\label{lem:LR}
                    For any $x\in\{0,1\}^{n_{x}}$, we have
                    \begin{equation}\label{eq:LR-strong}
                        \phi(x)=\min_{\lambda\in\mathbb{R}^{n_{x}}}L_{\ell}(x,\lambda).
                    \end{equation}
                    In addition, if $x\in X_{LF}$, there exists a $\lambda^{*}(x)\in\mathbb{R}^{n_{x}}$ such that for all $\lambda\in\mathbb{R}^{n_{x}}$ that satisfies
                    \begin{equation}\label{eq:LR-strong-condition}
                        \left\{\begin{aligned}
                            &\lambda_{i}\geq\lambda_{i}^{*}(x), \text{if~} x_{i}=1\\
                            &\lambda_{i}\leq\lambda_{i}^{*}(x), \text{if~} x_{i}=0
                        \end{aligned}\right.,~
                        \forall i\in[n_{x}],
                    \end{equation}
                    we have $\phi(x)=L_{\ell}(x,\lambda)$.
                \end{lemma}
                The proof of Lemma \ref{lem:LR} is similar to that of Lemma \ref{lem:penalty} and is thus placed in Appendix \ref{apd:lem:LR}.
                According to Lemma \ref{lem:LR}, when $x$ is feasible to $\phi$ and $\lambda_{i}$ is sufficiently large or sufficiently small (depending on the value of $x_{i}$), $L_{\ell}(x,\lambda)=\phi(x)$ holds and we have the following corollary.
                \begin{corollary}\label{crl:LR}
                    There exist constant vectors $U,L\in\mathbb{R}^{n_{x}}$ such that for any $x\in\{0,1\}^{n_{x}}\cap X_{LF}$ and  $\lambda\in\mathbb{R}^{n_{x}}$  that satisfies
                    \begin{equation}\label{eq:LR-strong-condition2}
                        \left\{\begin{aligned}
                            &\lambda_{i}\geq U_{i}, \text{if~} x_{i}=1\\
                            &\lambda_{i}\leq L_{i}, \text{if~} x_{i}=0
                        \end{aligned}\right.,~
                        \forall i\in[n_{x}],
                    \end{equation}
                    we have $\phi(x)=L_{\ell}(x,\lambda)$.
                \end{corollary}
                The proof of Corollary \ref{crl:LR} is similar to that of Corollary \ref{crl:penalty} and is thus placed in Appendix \ref{apd:crl:LR}.
                Based on Corollary \ref{crl:LR}, we derive a Lagrangian-based valid inequality in Proposition \ref{pps:LR-cut}.
                \begin{proposition}\label{pps:LR-cut}
                    For any $z\in[0,1]^{n_{x}}$, the following inequality is valid for any $(x,y) \in X \times Y$ satisfying \eqref{eq:optimality}:
                    \begin{equation}\label{eq:LR-cut}
                        d_{\ell}^{\top}y\geq \psi(z)-(\hat{\lambda}(1-x))^{\top}(x-z),
                    \end{equation}
                    where $\hat{\lambda}(x):=U\odot(1-x)+L\odot x$ and $\odot$ denotes the Hadamard product \cite{shapiro2021lectures}.
                    If $z\in\{0,1\}^{n_{x}}$, \eqref{eq:LR-cut} is valid and tight for \eqref{eq:optimality}, and it can be rewritten as
                    \begin{equation}\label{eq:LR-cut2}
                        d_{\ell}^{\top}y\geq \phi(z)-(\hat{\lambda}(z))^{\top}(x-z).
                    \end{equation}
                \end{proposition}
                The proof of Proposition \ref{pps:LR-cut} is similar to that of Proposition \ref{pps:penalty-cut} and is placed in Appendix \ref{apd:pps:LR-cut}.

                \begin{example}\label{exp:LR}
                    We continue Example \ref{exp:penalty} to illustrate the Lagrangian-based valid inequality \eqref{eq:LR-cut2}.
                    Because $x\in\{0,1\}^{n_{x}}$, the valid inequality in Example \ref{exp:penalty} can be recast as $d_{\ell}^{\top}y\geq \phi(z)-\rho\zeta(z)(x-z)$, where $\zeta(z)=1$ when $z=0$ and $\zeta(z)=-1$ when $z=1$.
                    This is equivalent to $d_{\ell}^{\top}y\geq \phi(z)-\lambda(z)(x-z)$ with $\lambda(z):=\rho\zeta(z)$.
                    Figure \ref{fig:example-LR} depicts and compares inequalities \eqref{eq:penalty-cut2} and \eqref{eq:LR-cut2} with $z=1$.
                    The black shaded part depicts the epigraph defined by the right-hand side of \eqref{eq:LR-cut2} with respect to a general $\lambda \leq L$, which coincides with the left side of the gray shaded part defined by \eqref{eq:penalty-cut2} (for $\rho = -\lambda$).
                    Note that the inequality \eqref{eq:LR-cut2} is tight at $x=1$.
                    As $\lambda$ approaches $-\infty$, the cut becomes valid.
                    This intuitively explains the existence of $\lambda^*(x)$ in Lemma \ref{lem:LR} and that of $L$ in Corollary \ref{crl:LR} because there are finitely many $x$-values.
                    In contrast, as $\lambda$ increases, the cut eventually touches the point $(0,\phi(0))$.
                    This pertains the threshold $L$, as depicted by the red solid line.
                    Intuitively, inequality \eqref{eq:LR-cut2} remains valid until $\lambda$ increases to $L$.
                    Therefore, $\lambda=L$ provides the strongest Lagrangian-based valid inequality with $z=1$.
                    \begin{figure}[!htbp]
                        \begin{center}
                            \vspace{-0ex}
                            \includegraphics[width=0.55\columnwidth]{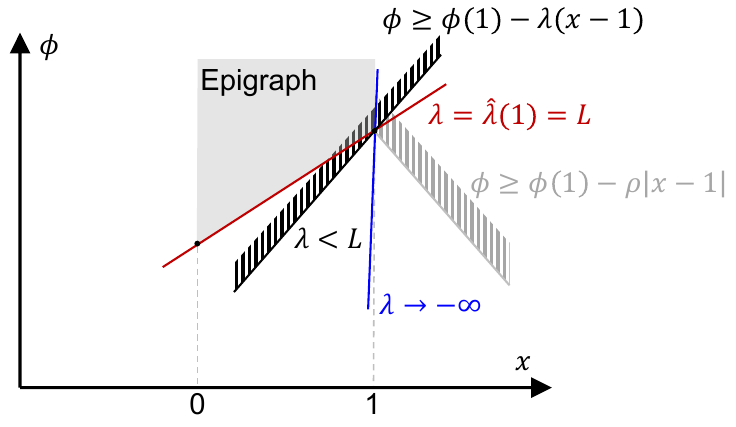}
                        \end{center}
                        \vspace{-4ex}
                        \caption{An illustrative example of the Lagrangian-based valid equality when $n_x=1$.}
                        \vspace{-2ex}
                        \label{fig:example-LR}
                    \end{figure}
                \end{example}

            \subsubsection{Augmented Lagrangian-based Valid Inequality}
                Furthermore, we consider the augmented Lagrangian method and design a Lagrangian function by relaxing $z=x$ as
                \begin{equation}\label{eq:ALR}
                    L_{a}(x, \lambda,\rho)=\max_{z\in[0,1]^{n_{x}}}~\psi(z)-\lambda^{\top}(x-z)-\rho\left(1^{\top}x+1^{\top}z-2x^{\top}z\right),
                \end{equation}
                where $\lambda\in\mathbb{R}^{n_{x}}$ and $\rho\in\mathbb{R}_{+}$ are Lagrangian multipliers.
                Obviously, for any $x\in\{0,1\}^{n_{x}}$, $z=x$ is feasible to the right-hand side of \eqref{eq:ALR}.
                Hence, $\phi(x)\leq L_{a}(x,\lambda,\rho)$ always holds, which implies the weak duality $\phi(x)\leq\min_{\lambda\in\mathbb{R}^{n_{x}},\rho\in\mathbb{R}_{+}}L_{a}(x,\lambda,\rho)$.
                Furthermore, because $x$ is binary-valued, $L_a(x, \lambda, \rho)$ admits the strong duality, as shown in Lemma \ref{lem:ALR}.
                \begin{lemma}\label{lem:ALR}
                    For any $x\in\{0,1\}^{n_{x}}$, we have
                    \begin{equation}\label{eq:ALR-strong}
                        \phi(x)=\min_{\lambda\in\mathbb{R}^{n_{x}},\rho\in\mathbb{R}_{+}}L_{a}(x,\lambda,\rho).
                    \end{equation}
                    In addition, if $x\in X_{LF}$, there exist $\lambda^{*}(x)\in\mathbb{R}^{n_{x}}$ and $\rho^{*}(x)\in\mathbb{R}_{+}$ such that for all $\rho\geq\rho^{*}(x)$ and for all $\lambda$ that satisfies
                    \begin{equation}\label{eq:ALR-strong-condition}
                        \left\{\begin{aligned}
                            &\lambda_{i}\geq\lambda_{i}^{*}(x), \text{if~} x_{i}=1\\
                            &\lambda_{i}\leq\lambda_{i}^{*}(x), \text{if~} x_{i}=0
                        \end{aligned}\right.,~
                        \forall i\in[n_{x}],
                    \end{equation}
                    we have $\phi(x)=L_{a}(x,\lambda,\rho)$.
                \end{lemma}
                The proof of Lemma \ref{lem:ALR} is similar to that of Lemma \ref{lem:penalty} and is thus placed in Appendix \ref{apd:lem:ALR}.
                According to Lemma \ref{lem:ALR}, $L_{a}(x,\lambda,\rho)=\phi(x)$ when $x$ is feasible to $\phi$, $\rho$ is sufficiently large, and $\lambda_{i}$ is sufficiently large or sufficiently small
                Hence, we have the following corollary.
                \begin{corollary}\label{crl:ALR}
                    There exist a constant $\hat{\rho}$ and constant vectors $U,L\in\mathbb{R}^{n_{x}}$ such that for any $x\in\{0,1\}^{n_{x}}\cap X_{LF}$, any $\rho\geq\hat{\rho}$, and any $\lambda\in\mathbb{R}^{n_{x}}$  that satisfies
                    \begin{equation}\label{eq:ALR-strong-condition2}
                        \left\{\begin{aligned}
                            &\lambda_{i}\geq U_{i}, \text{if~} x_{i}=1\\
                            &\lambda_{i}\leq L_{i}, \text{if~} x_{i}=0
                        \end{aligned}\right.,~
                        \forall i\in[n_{x}],
                    \end{equation}
                    we have $\phi(x)=L_{a}(x,\lambda,\rho)$.
                \end{corollary}
                The proof of Corollary \ref{crl:ALR} is similar to that of Corollary \ref{crl:penalty} and is thus placed in Appendix \ref{apd:crl:ALR}.

                Next, we derive an augmented Lagrangian-based valid inequality in Proposition \ref{pps:ALR-cut}.
                \begin{proposition}\label{pps:ALR-cut}
                    For any $z\in[0,1]^{n_{x}}$, the following inequality is valid for any $(x,y) \in X \times Y$ satisfying \eqref{eq:optimality}:
                    \begin{equation}\label{eq:ALR-cut}
                        d_{\ell}^{\top}y\geq \psi(z)-(\hat{\lambda}(1-x))^{\top}(x-z)-\hat{\rho}\left(1^{\top}x+1^{\top}z-2x^{\top}z\right).
                    \end{equation}
                    If $z\in\{0,1\}^{n_{x}}$, \eqref{eq:ALR-cut} is valid and tight for \eqref{eq:optimality}, and it can be rewritten as
                    \begin{equation}\label{eq:ALR-cut2}
                        d_{\ell}^{\top}y\geq \phi(z)-(\hat{\lambda}(z))^{\top}(x-z)-\hat{\rho}\left(1^{\top}x+1^{\top}z-2x^{\top}z\right).
                    \end{equation}
                \end{proposition}
                The proof of Proposition \ref{pps:ALR-cut} is similar to that of Proposition \ref{pps:penalty-cut} and is thus placed in Appendix \ref{apd:pps:ALR-cut}.
                When $\lambda=0$, \eqref{eq:ALR-cut2} reduces to \eqref{eq:penalty-cut2}; when $\rho=0$, \eqref{eq:ALR-cut2} reduces to \eqref{eq:LR-cut2}.

                \begin{example}\label{exp:ALR}
                    We continue using the same setup in Example \ref{exp:penalty}.
                    As compared to inequality \eqref{eq:penalty-cut2}, inequality \eqref{eq:ALR-cut2} adds a new term $-\lambda (x-z)$ on the right-hand side.
                    By adding the new term, we reshape the cut defined by inequality \eqref{eq:penalty-cut2} through changing the slopes on both sides.
                    Figure \ref{fig:example-ALR} shows an illustrative example with $z=1$.
                    The black shaded part depicts the epigraph defined by the right-hand side of \eqref{eq:ALR-cut2}.
                    Note that the inequality \eqref{eq:ALR-cut2} is tight at $x=1$ and its validity depends on the value of $\rho$ and $\lambda$.
                    By adjusting their values until the left side touches point $(0,\phi(0))$, we can obtain the strongest inequality.
                    Note that there may exist multiple \((\rho, \lambda)\) combinations to achieve this.
                    \begin{figure}[!htbp]
                        \begin{center}
                            \vspace{-0ex}
                            \includegraphics[width=0.5\columnwidth]{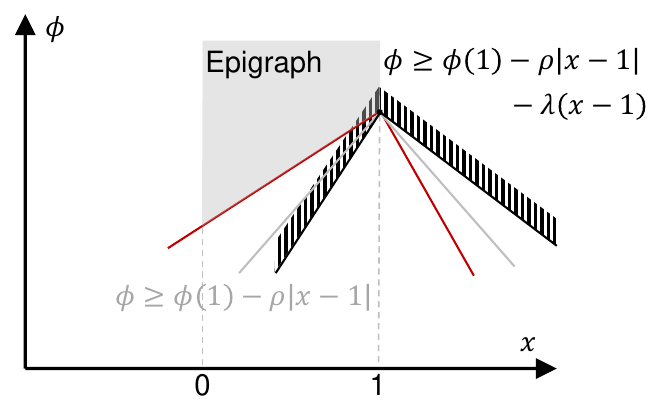}
                        \end{center}
                        \vspace{-4ex}
                        \caption{An illustrative example of the augmented Lagrangian-based valid equality when $n_{x}=1$.}
                        \vspace{-2ex}
                        \label{fig:example-ALR}
                    \end{figure}
                \end{example}

            \subsubsection{A Baseline Branch-and-Cut Algorithm}
                Thanks to the tightness of the Lagrangian-based valid inequalities \eqref{eq:penalty-cut2}, \eqref{eq:LR-cut2}, and \eqref{eq:ALR-cut2}, constraint \eqref{eq:optimality} can be represented by these inequalities, giving rise to a single-level MILP reformulation for the BP~\eqref{eq:bilevel}. We formalize this observation in Corollary \ref{crl:representation} and provide a proof in Appendix \ref{apd:crl:representation}.
                \begin{corollary}\label{crl:representation}
                    In the VFR \eqref{eq:VFR}, constraint \eqref{eq:optimality} can be fully described by a finite number of the Lagrangian-based inequalities \eqref{eq:penalty-cut2}, \eqref{eq:LR-cut2}, or \eqref{eq:ALR-cut2}.
                \end{corollary}

                Therefore, one can solve the BP~\eqref{eq:bilevel} by a standard branch-and-cut algorithm by separating constraint~\eqref{eq:optimality} using inequalities \eqref{eq:penalty-cut2}, \eqref{eq:LR-cut2}, or \eqref{eq:ALR-cut2} (e.g., through \texttt{lazy constraints} in Gurobi). For completeness, we state this in Algorithm~\ref{alg:bilevel} and its finite convergence to global optimum in the following theorem. 
                \begin{theorem}
                    The bilevel MILP model~\eqref{eq:bilevel} can be solved to global optimum by solving formulation \eqref{eq:VFR} using Algorithm~\ref{alg:bilevel}. 
                \end{theorem}

                \begin{algorithm}[!h]
                    \caption{A branch-and-cut algorithm using Lagrangian-based valid inequalities}
                    \label{alg:bilevel}
                    \KwIn{
                        Formulation \eqref{eq:upper objective}--\eqref{eq:lower feasibility}, Lagrangian coefficients $\hat{\rho}$ and $U/L$.
                    }
                    Initialize a queue $\mathcal{Q}$ for formulations, insert the continuous relaxation of \eqref{eq:upper objective}--\eqref{eq:lower feasibility} into $\mathcal{Q}$\;
                    Set upper bound $ub\leftarrow\infty$, optimal solution $(x^*,y^*)$\;
                    \While{$\mathcal{Q}$ is non-empty}{
                        Extract a formulation from $\mathcal{Q}$\;
                        Solve the formulation and obtain an incumbent solution $(\hat{x},\hat{y})$ with optimal value $\hat{f}$\;
                        \If{$\hat{f}<ub$ {\it and} $(\hat{x},\hat{y}) \in X \times Y$}{
                            Solve \eqref{eq:value function} to obtain $\phi(\hat{x})$\;
                            \If{$(\hat{x},\hat{y})$ violates \eqref{eq:optimality}}{
                                Incorporate valid inequality \eqref{eq:penalty-cut2}, \eqref{eq:LR-cut2}, or \eqref{eq:ALR-cut2} with $z=\hat{x}$ into the formulation and all formulations in $\mathcal{Q}$\;
                                Insert the formulation back into $\mathcal{Q}$\;
                            }
                            \Else{
                                Update $ub\leftarrow\hat{f}$ and $(x^*,y^*)\leftarrow(\hat{x},\hat{y})$\;
                            }
                        }
                        \ElseIf{$\hat{f}<ub$ and $(\hat{x},\hat{y}) \notin X \times Y$}{
                            Branch on $(\hat{x},\hat{y})$ and insert the two ensuring formulations into $\mathcal{Q}$\;
                        }
                    }
                    \KwOut{
                        Optimal solution $(x^*,y^*)$ and optimal value $ub$.
                    }
                \end{algorithm}

        \subsection{Selection of Coefficients}
            Though the above inequalities \eqref{eq:penalty-cut2}, \eqref{eq:LR-cut2}, and \eqref{eq:ALR-cut2} are all valid and tight for the lower-level optimality condition \eqref{eq:optimality}, their strengths are affected by coefficients $\hat{\rho}$, $U$, and $L$, which further influences the overall performance of the related branch-and-cut algorithm for solving the BP in \eqref{eq:bilevel}.
            We discuss how to select these coefficients in the following.

            \subsubsection{Exact Calculation}
                For brevity, we rewrite the Lagrangian relaxations, \eqref{eq:penalty}, \eqref{eq:LR}, and \eqref{eq:ALR}, in a general form as
                \begin{equation}\label{eq:GLR}
                    L(x, \lambda,\rho)=\max_{z\in[0,1]^{n_{x}}}~\bar{L}(x,\lambda,\rho,z),
                \end{equation}
                where $\bar{L}(x,\lambda,\rho,z)$ refers to the specific objective function in \eqref{eq:penalty}, \eqref{eq:LR}, or \eqref{eq:ALR}.
                \begin{lemma}\label{lem:solution}
                    The right-hand side of \eqref{eq:GLR} has a binary-valued optimal solution.
                \end{lemma}
                \begin{proof}\color{black}
                    We denote $z^{*}$ as the optimal solution to the right-hand side of \eqref{eq:GLR}.
                    Then, we suppose the contrary that there exists $i\in[n_{x}]$ such that $0<z_{i}^{*}<1$.
                    Note that $\psi(z)$ is linear in each individual $z_{i},~\forall i\in[n_{x}]$.
                    In \eqref{eq:penalty}, \eqref{eq:LR}, or \eqref{eq:ALR}, the penalty or Lagrangian term is also linear in each individual $z_{i},~\forall i\in[n_{x}]$.
                    Hence, we have
                    \begin{equation*}
                        \bar{L}(x,\lambda,\rho,z)=a(x,\lambda,\rho,z_{-i})z_{i}+b(x,\lambda,\rho,z_{-i})
                    \end{equation*}
                    where $a(\cdot)$ and $b(\cdot)$ are functions of $x$, $\lambda$, $\rho$, and $z_{-i}$.
                    We use $a(z_{-i})$ and $b(z_{-i})$ for brevity.\\
                    When $a(z^{*}_{-i})>0$, because $z_{i}^{*}<1$, we have
                    \begin{equation*}
                        \bar{L}(x,\lambda,\rho,z_{-i}^{*},z_{i}^{*})=a(z_{-i}^{*})z_{i}+b(z_{-i}^{*})<a(z_{-i}^{*})\times1+b(z_{-i}^{*})=\bar{L}(x,\lambda,\rho,z_{-i}^{*},1),
                    \end{equation*}
                    which contradicts with that $z^{*}$ is optimal.\\
                    When $a(z^{*}_{-i})<0$, because $z_{i}^{*}>0$, we have
                    \begin{equation*}
                        \bar{L}(x,\lambda,\rho,z_{-i}^{*},z_{i}^{*})=a(z_{-i}^{*})z_{i}+b(z_{-i}^{*})<a(z_{-i}^{*})\times0+b(z_{-i}^{*})=\bar{L}(x,\lambda,\rho,z_{-i}^{*},0),
                    \end{equation*}
                    which contradicts with that $z^{*}$ is optimal.\\
                    When $a(z^{*}_{-i})=0$, we have
                    \begin{equation*}
                        \bar{L}(x,\lambda,\rho,z_{-i}^{*},z_{i}^{*})=b(z_{-i}^{*})=\bar{L}(x,\lambda,\rho,z_{-i}^{*},0)=\bar{L}(x,\lambda,\rho,z_{-i}^{*},1),
                    \end{equation*}
                    which contradicts with that $z_{i}^{*}\notin\{0,1\}$.
                    This completes the proof.
                \end{proof}

                Based on Lemma \ref{lem:solution}, we state the selection of $\hat{\rho}$ and $U/L$ in Proposition \ref{pps:selection}, whose proof is placed in Appendix \ref{apd:pps:selection}.
                \begin{proposition}\label{pps:selection}
                    The following $\hat{\rho}$ and $U/L$ can be used to construct corresponding valid inequalities:\\
                    (1) To construct the penalty-based valid inequality \eqref{eq:penalty-cut2}, the following $\hat{\rho}$ is sufficiently large:
                    \begin{equation}\label{eq:selection-penalty}
                        \begin{aligned}
                            \hat{\rho}=\max_{z,z'\in\{0,1\}^{n_{x}}}&~\phi(z)-\phi(z')\\
                            \text{s.t.}&~\|z-z'\|_{2}^{2}=1.
                        \end{aligned}
                    \end{equation}
                    (2) To construct the Lagrangian-based valid inequality \eqref{eq:LR-cut2}, the following $U$ and $L$ are sufficiently large and sufficiently small, respectively:
                    \begin{equation}\label{eq:selection-LR}
                        \text{For all}~i\in[n_{x}],
                        \left\{\begin{aligned}
                            \begin{aligned}
                                L_{i}=\min_{z,z'\in\{0,1\}^{n_{x}}}&~\phi(z)-\phi(z')\\
                                \text{s.t.}&~z_{-i}=z'_{-i},z_{i}=0,z'_{i}=1
                            \end{aligned}\\
                            \begin{aligned}
                                U_{i}=\max_{z,z'\in\{0,1\}^{n_{x}}}&~\phi(z)-\phi(z')\\
                                \text{s.t.}&~z_{-i}=z'_{-i},z_{i}=0,z'_{i}=1
                            \end{aligned}
                        \end{aligned}\right.
                    \end{equation}
                    (3) To construct the augmented Lagrangian-based valid inequality \eqref{eq:ALR-cut2}, any $\hat{\rho}\geq0$ and the following $U$ and $L$ are sufficient:
                    \begin{equation}\label{eq:selection-ALR}
                        \text{For all}~i\in[n_{x}],
                        \left\{\begin{aligned}
                            \begin{aligned}
                                L_{i}=\hat{\rho}+\min_{z,z'\in\{0,1\}^{n_{x}}}&~\phi(z)-\phi(z')\\
                                \text{s.t.}&~z_{-i}=z'_{-i},z_{i}=0,z'_{i}=1
                            \end{aligned}\\
                            \begin{aligned}
                                U_{i}=-\hat{\rho}+\max_{z,z'\in\{0,1\}^{n_{x}}}&~\phi(z)-\phi(z')\\
                                \text{s.t.}&~z_{-i}=z'_{-i},z_{i}=0,z'_{i}=1
                            \end{aligned}
                        \end{aligned}\right.
                    \end{equation}
                \end{proposition}

                If we select $\hat{\rho}$ and $U/L$ according to Proposition \ref{pps:selection}, the three types of valid inequalities can be connected through the following corollary.
                \begin{corollary}\label{crl:relationship}
                    When we select $\hat{\rho}$ and $U/L$ according to Proposition \ref{pps:selection}, the following claims hold:\\
                    (i) the Lagrangian-based inequality \eqref{eq:LR-cut2} is stronger than or equivalent to the penalty-based inequality \eqref{eq:penalty-cut2}, and\\
                    (ii) the augmented Lagrangian-based inequality \eqref{eq:ALR-cut2} is equivalent to the Lagrangian-based inequality \eqref{eq:LR-cut2}.
                \end{corollary}
                \begin{proof}\color{black}
                    Comparing \eqref{eq:selection-penalty} and \eqref{eq:selection-LR}, we have
                    \begin{equation*}
                        \hat{\rho}=\max_{i\in[n_{x}]}\{\max\{-L_{i},U_{i}\}\}.
                    \end{equation*}
                    Hence, the right-hand side of \eqref{eq:LR-cut2} can be relaxed as
                    \begin{equation*}
                        \begin{aligned}
                            \phi(z)-(\hat{\lambda}(z))^{\top}(x-z)=&\phi(z)-\sum_{i\in[n_{x}]}\big(U_{i}(1-z_{i})+L_{i}z_{i}\big)(x_{i}-z_{i})\\
                            \geq&\phi(z)-\sum_{i\in[n_{x}]}\big(\hat{\rho}(1-z_{i})-\hat{\rho}z_{i}\big)(x_{i}-z_{i})\\
                            =&\phi(z)-\hat{\rho}\left(1^{\top}x+1^{\top}z-2x^{\top}z\right).
                        \end{aligned}
                    \end{equation*}
                    The above inequality holds because
                    \begin{equation*}
                        \big(\hat{\rho}(1-z_{i})-\hat{\rho}z_{i}\big)(x_{i}-z_{i})-\big(U_{i}(1-z_{i})+L_{i}z_{i}\big)(x_{i}-z_{i})=(\hat{\rho}-U_{i})x_{i}(1-z_{i})+(\hat{\rho}+L_{i})(1-x_{i})z_{i}\geq0.
                    \end{equation*}
                    Hence, \eqref{eq:LR-cut2} at least as strong as \eqref{eq:penalty-cut2}.

                    To distinguish $U/L$ in \eqref{eq:selection-LR} and \eqref{eq:selection-ALR}, we denote $U/L$ in \eqref{eq:selection-LR} as $U_{\ell}/L_{\ell}$ and denote $U/L$ in \eqref{eq:selection-ALR} as $U_{a}/L_{a}$.
                    Comparing \eqref{eq:selection-LR} and \eqref{eq:selection-ALR}, we can see that $U_{a}+\hat{\rho}=U_{\ell}$ and $L_{a}-\hat{\rho}=L_{\ell}$.
                    Then, we expand the right-hand side of \eqref{eq:ALR-cut2} as
                    \begin{equation*}
                        \begin{aligned}
                            &\phi(z)-\big(U_{a}\odot(1-z)+L_{a}\odot z\big)(x-z)-\hat{\rho}\left(1^{\top}x+1^{\top}z-2x^{\top}z\right)\\
                            =&\phi(z)-\big((U_{a}+\hat{\rho})\odot(1-z)+(L_{a}-\hat{\rho})\odot z\big)(x-z)\\
                            =&\phi(z)-\big(U_{\ell}\odot(1-z)+L_{\ell}\odot z\big)(x-z).
                        \end{aligned}
                    \end{equation*}
                    Hence, \eqref{eq:ALR-cut2} is equivalent to \eqref{eq:LR-cut2}.
                \end{proof}

                In view of claim (ii) in Corollary \ref{crl:relationship}, we only discuss the penalty-based valid inequality \eqref{eq:penalty-cut2} and the Lagrangian-based valid inequality \eqref{eq:LR-cut2} in the remainder of the paper.

            \subsubsection{Quick Calculation}
                In practical implementation, deriving $\hat{\rho}$ or $U/L$ using Proposition \ref{pps:selection} requires solving max-min or min-max problems, which is computationally expensive or even intractable when $y$ involves integer variables.
                In light of this issue, we adopt relaxations for tractable and quick calculation of $\hat{\rho}$ and $U/L$, which is presented in Proposition \ref{pps:selection-quick}, whose proof is placed in Appendix \ref{apd:pps:selection-quick}.
                \begin{proposition}\label{pps:selection-quick}
                    The following $\hat{\rho}$ and $U/L$ can be used to construct corresponding valid inequalities:\\
                    (1) To construct the penalty-based valid inequality \eqref{eq:ALR-cut2}, the following $\hat{\rho}$ is sufficiently large.
                    \begin{equation}\label{eq:selection-penalty-quick}
                        \begin{aligned}
                            \hat{\rho}=\max~&d_{\ell}^{\top}y-d_{\ell}^{\top}y'\\
                            \text{s.t.~}&z,z'\in\{0,1\}^{n_{x}},y,y'\in Y\\
                            &1^{\top}z+1^{\top}z'-2\times1^{\top}\gamma=1\\
                            &0\leq\gamma\leq z,~z+z'-1\leq\gamma\leq z'\\
                            &B_{\ell}y\leq h_{\ell}-A_{\ell}z\\
                            &B_{\ell}y'\leq h_{\ell}-A_{\ell}z'
                        \end{aligned}
                    \end{equation}
                    Here, $\gamma\in\mathbb{R}^{n_{x}}$ is an auxiliary variable.\\
                    (2) To construct the Lagrangian-based valid inequality \eqref{eq:LR-cut2}, the following $U$ and $L$ are sufficiently large and sufficiently small, respectively.
                    \begin{equation}\label{eq:selection-LR-quick}
                        \text{For all}~i\in[n_{x}],
                        \left\{\begin{aligned}
                            \begin{aligned}
                                L_{i}=\min_{z,z'\in\{0,1\}^{n_{x}},y,y'\in Y}&~d_{\ell}^{\top}y-d_{\ell}^{\top}y'\\
                                \text{s.t.~}&~z_{-i}=z'_{-i},z_{i}=0,z'_{i}=1\\
                                &B_{\ell}y\leq h_{\ell}-A_{\ell}z,B_{\ell}y'\leq h_{\ell}-A_{\ell}z'
                            \end{aligned}\\
                            \begin{aligned}
                                U_{i}=\max_{z,z'\in\{0,1\}^{n_{x}},y,y'\in Y}&~d_{\ell}^{\top}y-d_{\ell}^{\top}y'\\
                                \text{s.t.~}&~z_{-i}=z'_{-i},z_{i}=0,z'_{i}=1\\
                                &B_{\ell}y\leq h_{\ell}-A_{\ell}z,B_{\ell}y'\leq h_{\ell}-A_{\ell}z'
                            \end{aligned}
                        \end{aligned}\right.
                    \end{equation}
                \end{proposition}

                Through Proposition \ref{pps:selection-quick}, the coefficients can be calculated by solving MILPs and can be found by off-the-shelf solvers, such as Gurobi.
                Though the calculated $\hat{\rho}$ and $U/L$ from Proposition \ref{pps:selection-quick} can also lead to valid and tight inequalities, their strengths can be weaker than those from Proposition \ref{pps:selection}.
                In the following sections, we will discuss how to quickly find valid and strong coefficients under special and general problem structures.

    \section{Special Property-based Valid Inequalities}\label{sec:special}
        Real-world sequential optimization problems often have specific structures, some of which may induce special properties and help the computation of related bilevel MILPs.
        In this section, we focus on cases where the lower-level value function $\phi(x)$ has such special properties and explore the benefit from these special properties in solving BPs.

        In light of our Assumption \ref{asp:binary} that all linking variables are binary-valued, submodularity or supermodularity is a promising property in solving integer programs \cite{qi2021sequential}.
        We recall the definitions of submodularity and supermodularity as follows.
        \begin{definition}[\cite{simchi2014convexity}]\label{def:property}
            Consider a function $f: \{0,1\}^{n_{x}} \rightarrow \mathbb{R}$.
            Then, \\
            (i) $f$ is called submodular if $f(x') + f(x'') \geq f(x' \vee x'') + f(x' \wedge x'')$ for all $x', x'' \in \{0,1\}^{n_{x}} $.\\
            (ii) $f$ is called supermodular if $f(x') + f(x'') \leq f(x' \vee x'') + f(x' \wedge x'')$ for all $x', x'' \in \{0,1\}^{n_{x}} $.
        \end{definition}
        Refs. \cite{chen2021preservation, long2024supermodularity} have identified many applications with submodularity or supermodularity, including the newsvendor problem, uncapacitated facility location, lot sizing, appointment scheduling, and assemble-to-order.
        By exploiting the submodularity or supermodularity of $\phi(x)$, we strengthen the above Lagrangian-based valid inequalities. 

        \subsection{Efficient and Exact Calculation of Lagrangian Coefficients}
            When $\phi(x)$ is submodular or supermodular, the calculation of $\hat{\rho}$ and $U_{i}/L_{i}$ in Proposition \ref{pps:selection} can be significantly simplified, as presented in Proposition \ref{pps:selection-special}, whose proof is placed in Appendix \ref{apd:pps:selection-special}.
            \begin{proposition}\label{pps:selection-special}
                The following $\hat{\rho}$ and $U/L$ are sufficient to construct the penalty-based valid inequality \eqref{eq:penalty-cut2} or the Lagrangian-based valid inequalities \eqref{eq:LR-cut2}:\\
                (1) If $\phi(x)$ is submodular in $x\in\{0,1\}^{n_{x}}$, then for any $i\in[n_{x}]$, formulations \eqref{eq:selection-penalty} and \eqref{eq:selection-LR} admit the following closed-form solutions:
                \begin{equation}\label{eq:selection-special-sub}
                    L_{i}=\phi(0)-\phi(e_{i}),~
                    U_{i}=\phi(1-e_{i})-\phi(1),~
                    \hat{\rho}=\max_{i\in[n_{x}]}\left\{\max\{-L_{i},U_{i}\}\right\}.
                \end{equation}
                (2) If $\phi(x)$ is supermodular in $x\in\{0,1\}^{n_{x}}$, then for any $i\in[n_{x}]$, formulations \eqref{eq:selection-penalty} and \eqref{eq:selection-LR} admit the following closed-form solutions:
                \begin{equation}\label{eq:selection-special-super}
                    L_{i}=\phi(1-e_{i})-\phi(1),~
                    U_{i}=\phi(0)-\phi(e_{i}),~
                    \hat{\rho}=\max_{i\in[n_{x}]}\left\{\max\{-L_{i},U_{i}\}\right\}.
                \end{equation}
                Here, $e_{i}\in\mathbb{R}^{n_{x}}$ is a vector with entry $i$ being 1 and all other entries being 0.
            \end{proposition}

            Either submodularity or supermodularity of $\phi(x)$ waives the need of solving complex min-max or max-min optimization and only needs to solve $2n_{x}+2$ lower-level problems, which is much more efficient than the calculation in Proposition \ref{pps:selection}.
            Meanwhile, the calculation in Proposition \ref{pps:selection-special} is equivalent to that in Proposition \ref{pps:selection}, and hence, valid inequalities through Proposition \ref{pps:selection-special} have the same strengths with those through Proposition \ref{pps:selection}, providing stronger inequalities than those through Proposition \ref{pps:selection-quick}.

        \subsection{Submodularity/Supermodularity-based Valid Inequalities}
            Except for strengthening Lagrangian-based valid inequalities, submodularity or supermodularity also enables a new family of valid inequalities to describe $\phi(x)$, leading to Propositions \ref{pps:submodular-cut} and \ref{pps:supermodular-cut}.
            Before we get into the propositions, we define a mapping $\phi:2^{[n_{x}]}\rightarrow\mathbb{R}$ as
            \begin{equation}
                \phi(\mathcal{S}):=\phi(s),
            \end{equation}
            where $s\in\{0,1\}^{n_{x}}$ and for any $i\in[n_{x}]$, $s_{i}=1$ if $i\in\mathcal{S}$ and $s_{i}=0$ if $i\notin\mathcal{S}$.

            \begin{proposition}[Theorem 1 in \cite{shen2023chance}]\label{pps:submodular-cut}
                If $\phi(x)$ is submodular in $x\in\{0,1\}^{n_{x}}$, for any $z\in\{0,1\}^{n_{x}}$, we have a valid and tight inequality for \eqref{eq:optimality} as
                \begin{equation}\label{eq:submodular-cut}
                    d_{\ell}^{\top}y\geq \phi(\mathcal{S}_{0})+\sum_{k=1}^{n_{x}}[\phi(\mathcal{S}_{k})-\phi(\mathcal{S}_{k-1})]x_{\sigma_{k}}.
                \end{equation}
                where $\sigma$ is a permutation of $[n_{x}]$ such that $z_{\sigma_{1}}\geq z_{\sigma_{2}}\geq\cdots\geq z_{\sigma_{n_{x}}}$;
                $\mathcal{S}_{k}:=\{\sigma_{1},\ldots,\sigma_{k}\}$ defines the former $k$ entries of $\sigma$ and $\mathcal{S}_{0}:=\emptyset$.
            \end{proposition}

            \begin{proposition}[Theorem 6 in \cite{nemhauser1981maximizing}]\label{pps:supermodular-cut}
                If $\phi(x)$ is supermodular in $x\in\{0,1\}^{n_{x}}$, for any $z\in\{0,1\}^{n_{x}}$, we have a valid and tight inequality for \eqref{eq:optimality} as
                \begin{equation}\label{eq:supermodular-cut}
                    d_{\ell}^{\top}y\geq \phi(\mathcal{S}_{z})-\sum_{i\in \mathcal{S}_{z}}\delta([n_{x}]\backslash\{i\},\{i\})(1-x_{i})+\sum_{i\in[n_{x}]\backslash \mathcal{S}_{z}}\delta(\mathcal{S}_{z},\{i\})x_{i}.
                \end{equation}
                where $\mathcal{S}_{z}:=\{i\in[n_{x}]:z_{i}=1\}$ and for any $\mathcal{S}\subseteq[n_{x}]$ and $i\in[n_{x}]\backslash\mathcal{S}$, $\delta(\mathcal{S},\{i\}):=\phi(\mathcal{S}\cup\{i\})-\phi(\mathcal{S})$.
            \end{proposition}

            \begin{example}\label{exp:Lagrangian}
                Figure \ref{fig:example-special} shows two-dimensional examples of the submodularity/supernodularity based valid equalities.
                In the left diagram, we have $\phi(1,0)+\phi(0,1)\geq\phi(0,0)+\phi(1,1)$, which means that $\phi$ is submodular in $x$.
                According to \eqref{eq:submodular-cut}, the submodularity-based valid inequality at $z=(1,1)$ is depicted.
                We also depict the Lagrangian-based and penalty-based valid inequalities at $z=(1,1)$, using the exact coefficients from Proposition \ref{pps:selection}.
                We can see that all these valid inequalities are tight at $(x_{1},x_{2})=(1,1)$.
                The submodularity-based valid inequality constructs a facet-defining cut and is the strongest one among the three valid inequalities.
                The Lagrangian-based valid inequality is stronger than the penalty-based valid inequality, which coincides with Corollary \ref{crl:relationship}.
                In the right diagram, we have $\phi(1,0)+\phi(0,1)\leq\phi(0,0)+\phi(1,1)$, which means $\phi$ is supermodular in $x$.
                At $z=(1,0)$, the supermodularity-based valid inequality by \eqref{eq:supermodular-cut} and the Lagrangian-based and penalty-based valid inequalities with the exact coefficients are depicted.
                We have similar observations that the supermodularity-based valid inequality is facet-defining and strongest, and the Lagrangian-based valid inequality is stronger than the penalty-based valid inequality.
                \begin{figure}[!htbp]
                    \begin{center}
                        \vspace{-0ex}
                        \includegraphics[width=0.48\columnwidth]{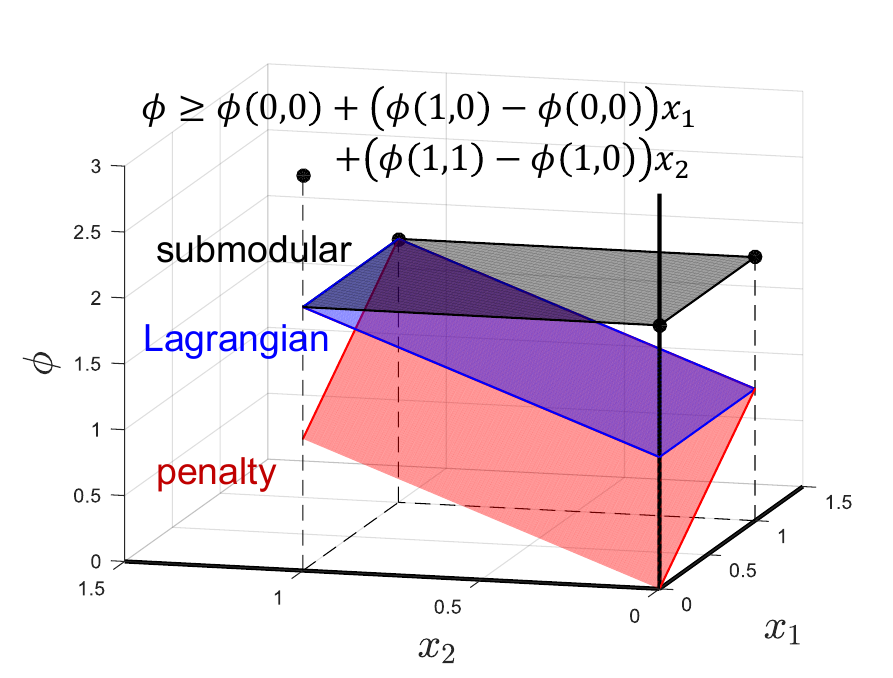}
                        \includegraphics[width=0.48\columnwidth]{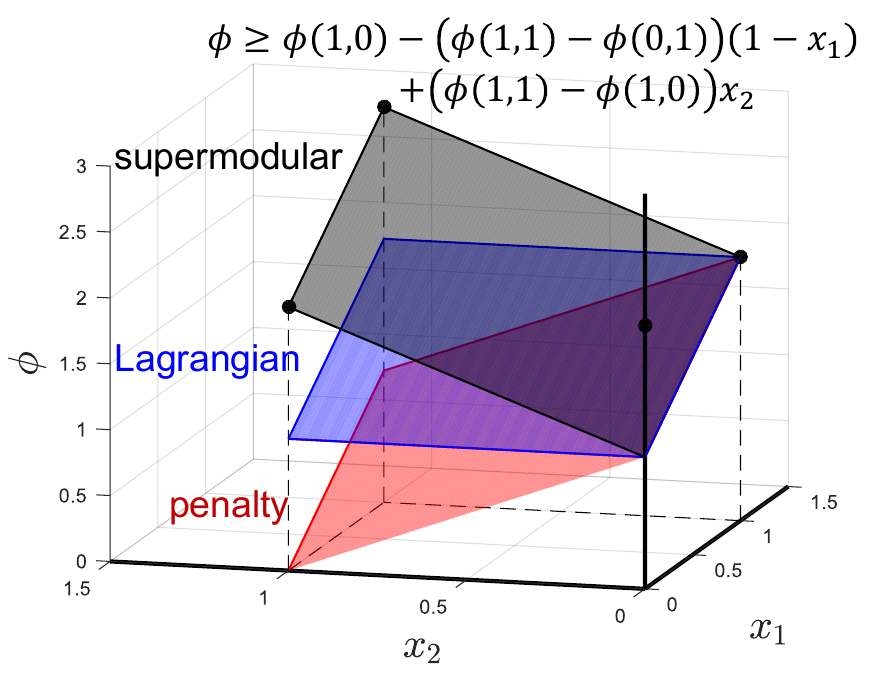}
                    \end{center}
                    \vspace{-4ex}
                    \caption{Illustrative examples of the submodularity/supermodularity-based valid equalities.\\(Left: submodular, $z=(1,1)$; Right: supernodular, $z=(1,0)$)}
                    \vspace{-2ex}
                    \label{fig:example-special}
                \end{figure}
            \end{example}

            Constructing \eqref{eq:submodular-cut} or \eqref{eq:supermodular-cut} only needs to solve several lower-level problems.
            For each \eqref{eq:submodular-cut}, we need to solve $n_{x}+1$ lower-level problems.
            For each \eqref{eq:supermodular-cut}, we need to solve $n_{x}+1$ lower-level problems if $z=1$ and $n_{x}+2$ lower-level problems if not.
            In practical implementation, when we add a new valid inequality regarding a new $z$, there may be repeated calculations, such as $\phi(0)$ in \eqref{eq:submodular-cut} and $\phi(1)$ in \eqref{eq:supermodular-cut}.
            In light of this, we can set a pool to save calculated $\phi$ values and search the pool before calculating any new $\phi(x)$, which avoids repeated calculations and improves computational efficiency.
            On the other hand, when $\phi(x)$ is submodular or supermodular, the right-hand side of \eqref{eq:submodular-cut} or \eqref{eq:supermodular-cut} constructs a facet-defining cut of $\phi(x)$~\cite{schrijver2003combinatorial, nemhauser1981maximizing}.
            Hence, \eqref{eq:submodular-cut} or \eqref{eq:supermodular-cut} shows higher strengths than Lagrangian-based valid inequalities and helps to enhance computational efficiency, which we will demonstrate numerically in Section \ref{sec:results}.

        \subsection{Quasi-Submodularity/Quasi-Supermodularity}
            Though submodularity or supermodularity leads to strong valid inequalities, the property does not always exist or is hard to identify.
            In light of this, we further extend to cases where $\phi(x)$ may not possess submodularity or supermodularity but recover such properties after fixing the integer decision variables at the lower level.
            To define this property precisely, we recall the value function \eqref{eq:value function} as
            \begin{align}\label{eq:quasi}
                \phi(x)=\max_{y_{1}\in Y_{1}}\left\{d_{l1}^{\top}y_{1}+\varphi(x,y_{1})\right\},
            \end{align}
            where $y=[y_{1}^{\top},y_{2}^{\top}]^\top$, $Y=Y_{1}\times Y_{2}$, and
            \begin{equation}\label{eq:varphi}
                \begin{aligned}
                    \varphi(x,y_{1}):=\max_{y_{2}\in Y_{2}}&~d_{l2}^{\top}y_{2}\\
                    \text{s.t.}&~B_{l2}y_{2}\leq h_{\ell}-A_{\ell}x-B_{l1}y_{1}.
                \end{aligned}
            \end{equation}
            \begin{definition}\label{def:quasi}
                Consider the function $\phi: \{0,1\}^{n_{x}}\rightarrow \mathbb{R}$ as defiend in \eqref{eq:quasi}.
                Then,\\
                (i) $\phi$ is called \emph{quasi}-submodular, if for any $y_{1}\in Y_{1}$, $\varphi(x,y_{1})$ is submodular in $x$.\\
                (ii) $\phi$ is called \emph{quasi}-supermodular, if for any $y_{1}\in Y_{1}$, $\varphi(x,y_{1})$ is supermodular in $x$.
            \end{definition}

            Quasi-submodularity/quasi-supermodularity is more common than submodularity/supermodularity and is easier to identify.
            For example, suppose that $Y_{1}\subseteq\{0,1\}^{n_{y_{1}}}$ and $Y_{2}\subseteq\mathbb{R}^{n_{y_{2}}}$, where $n_{y_{1}}$ and $n_{y_{2}}$ are the numbers of binary variables and continuous variables in $y$, respectively (namely, $n_{y_{1}}+n_{y_{2}}=n_{y}$).
            Then, $\varphi(x,y_{1})$ defined in \eqref{eq:varphi} is a linear program, and we can use the identified conditions in \cite{chen2021preservation, long2024supermodularity} to check whether $\varphi(x,y_{1})$ is submodular/supermodular in $x$.
            We note that \cite{chen2021preservation} also identifies conditions to directly check the supermodularity of $\phi(x)$, yet such conditions are impractical when $y$ involves discrete variables.

            \subsubsection{Lagrangian-based Valid Inequalities with Exact Coefficients}
                For the value function defined in \eqref{eq:quasi}, the Lagrangian-based methods are also applicable and we design the Lagrangian-based valid inequality in Proposition \ref{pps:quasi-Lagrangian-cut}.
                \begin{proposition}\label{pps:quasi-Lagrangian-cut}
                    For any $z\in\{0,1\}^{n_{x}}$, the following inequality is valid and tight for \eqref{eq:optimality}:
                    \begin{equation}\label{eq:quasi-Lagrangian-cut}
                        d_{\ell}^{\top}y\geq \phi(z)-(\hat{\lambda}(z,\hat{y}_{1}))^{\top}(x-z),
                    \end{equation}
                    where $\hat{y}_{1}\in\arg\max_{y_{1}\in Y_{1}}\left\{d_{l1}^{\top}y_{1}+\varphi(z,y_{1})\right\}$ and $\hat{\lambda}(z,\hat{y}_{1}):=U(\hat{y}_{1})\odot(1-z)+L(\hat{y}_{1})\odot z$.
                    Here, $U(\hat{y}_{1})$ and $L(\hat{y}_{1})$ are sufficiently large and sufficiently small vectors, respectively, and their exact calculation depends on the value of $\hat{y}_{1}$.
                \end{proposition}
                \begin{proof}\color{black}
                    From \eqref{eq:quasi}, we have
                    \begin{align*}
                        \phi(x)\geq d_{l1}^{\top}y_{1}+\varphi(x,y_{1}),~\forall y_{1}\in Y_{1},
                    \end{align*}
                    which implies
                    \begin{align*}
                        d_{\ell}^{\top}y\geq d_{l1}^{\top}\hat{y}_{1}+\varphi(x,\hat{y}_{1})
                    \end{align*}
                    is valid for \eqref{eq:optimality}.
                    We follow the derivation of \eqref{eq:LR-cut2} to handle $\varphi(x,\hat{y}_{1})$, and hence,
                    \begin{align*}
                        d_{\ell}^{\top}y\geq d_{l1}^{\top}\hat{y}_{1}+\varphi(z,\hat{y}_{1})-(\hat{\lambda}(z))^{\top}(x-z)
                    \end{align*}
                    is valid for \eqref{eq:optimality}, where $\hat{\lambda}(z)=U\odot(1-z)+L\odot z$.
                    Because this valid inequality is obtained under a fixed $\hat{y}_{1}$, the coefficients $U$ and $L$ should depend on the value of $\hat{y}_{1}$.
                    So we denote $\hat{\lambda}(z,\hat{y}_{1})=U(\hat{y}_{1})\odot(1-z)+L(\hat{y}_{1})\odot z$.
                    Furthermore, because $\hat{y}_{1}\in\arg\max_{y_{1}\in Y_{1}}\left\{d_{l1}^{\top}y_{1}+\varphi(z,y_{1})\right\}$, we have $d_{l1}^{\top}\hat{y}_{1}+\varphi(z,\hat{y}_{1})=\phi(z)$.
                    Therefore, \eqref{eq:quasi-Lagrangian-cut} is valid for \eqref{eq:optimality}.

                    By fixing $x$ on the right-hand side of \eqref{eq:quasi-Lagrangian-cut} as $z$, we obtain
                    \begin{equation*}
                        d_{\ell}^{\top}y\geq \phi(z),
                    \end{equation*}
                    which implies that \eqref{eq:quasi-Lagrangian-cut} is tight for \eqref{eq:optimality}.
                    This completes the proof.
                \end{proof}

                We note that \eqref{eq:quasi-Lagrangian-cut} is similar to \eqref{eq:LR-cut2} and the only difference lies in the exact calculation of Lagrangian coefficients.
                The exact calculation for \eqref{eq:LR-cut2} follows from Proposition \ref{pps:selection}, which involves min-max or max-min optimization and is intractable.
                However, with quasi-submodularity or quasi-supermodularity, we can first solve $\hat{y}_{1}$ and then follow Proposition \ref{pps:selection-special} to calculate the exact coefficients for \eqref{eq:quasi-Lagrangian-cut}, which is more computationally efficient.
                In the implementation of \eqref{eq:quasi-Lagrangian-cut}, although we need to solve new Lagrangian coefficients once we have obtained a new $\hat{y}_{1}$, the induced valid inequalities are strong.

            \subsubsection{Quasi-Submodularity/Quasi-Supermodularity-based Valid Inequalities}
                With quasi-submodularity or quasi-supermodularity, we design new valid inequalities in Propositions \ref{pps:quasi-submodular-cut} and \ref{pps:quasi-supermodular-cut}.
                We define a mapping $\varphi:2^{[n_{x}]}\times Y_1\rightarrow\mathbb{R}$ as
                \begin{equation}
                    \varphi(\mathcal{S},y_{1}):=\varphi(s,y_{1}),
                \end{equation}
                where $s\in\{0,1\}^{n_{x}}$ and for any $i\in[n_{x}]$, $s_{i}=1$ if $i\in\mathcal{S}$ and $s_{i}=0$ if $i\notin\mathcal{S}$.
                \begin{proposition}\label{pps:quasi-submodular-cut}
                    If $\phi(x)$ defined as \eqref{eq:quasi} is quasi-submodular in $x\in\{0,1\}^{n_{x}}$, for any $z\in\{0,1\}^{n_{x}}$, we have a valid and tight inequality for \eqref{eq:optimality} as
                    \begin{equation}\label{eq:quasi-submodular-cut}
                        d_{\ell}^{\top}y\geq d_{l1}^{\top}\hat{y}_{1}+\varphi(\mathcal{S}_{0},\hat{y}_{1})+\sum_{k=1}^{n_{x}}[\varphi(\mathcal{S}_{k},\hat{y}_{1})-\varphi(\mathcal{S}_{k-1},\hat{y}_{1})]x_{\sigma_{k}}.
                    \end{equation}
                    where $\hat{y}_{1}\in\arg\max_{y_{1}\in Y_{1}}\left\{d_{l1}^{\top}y_{1}+\varphi(z,y_{1})\right\}$;
                    $\sigma$ is a permutation of $[n_{x}]$ such that $z_{\sigma_{1}}\geq z_{\sigma_{2}}\geq\cdots\geq z_{\sigma_{n_{x}}}$;
                    $\mathcal{S}_{k}:=[\sigma_{1},\ldots,\sigma_{k}]$ defines the former $k$ entries of $\sigma$ and $\mathcal{S}_{0}:=\emptyset$.
                \end{proposition}

                \begin{proposition}\label{pps:quasi-supermodular-cut}
                    If $\phi(x)$ defined as \eqref{eq:quasi} is quasi-supermodular in $x\in\{0,1\}^{n_{x}}$, for any $x\in\{0,1\}^{n_{x}}$, we have a valid and tight inequality for \eqref{eq:optimality} as
                    \begin{equation}\label{eq:quasi-supermodular-cut}
                        d_{\ell}^{\top}y\geq d_{l1}^{\top}\hat{y}_{1}+\varphi(\mathcal{S},\hat{y}_{1})-\sum_{i\in \mathcal{S}}\delta([n_{x}]\backslash\{i\},\{i\},\hat{y}_{1})(1-x_{i})+\sum_{i\in[n_{x}]\backslash \mathcal{S}}\delta(\mathcal{S},\{i\},\hat{y}_{1})x_{i}.
                    \end{equation}
                    where $\hat{y}_{1}\in\arg\max_{y_{1}\in Y_{1}}\left\{d_{l1}^{\top}y_{1}+\varphi(z,y_{1})\right\}$;
                    $\mathcal{S}_{z}=\{i\in[n_{x}]:z_{i}=1\}$ and for any $\mathcal{S}\subseteq[n_{x}]$, $i\in[n_{x}]\backslash\mathcal{S}$, $y_1\in Y_1$, $\delta(\mathcal{S},\{i\},y_{1}):=\varphi(\mathcal{S}\cup\{i\},y_{1})-\varphi(\mathcal{S},y_{1})$.
                \end{proposition}
                The proofs of Propositions \ref{pps:quasi-submodular-cut} and \ref{pps:quasi-supermodular-cut} are similar to that of Proposition \ref{pps:quasi-Lagrangian-cut} and thus omitted.

                Constructing \eqref{eq:quasi-submodular-cut} or \eqref{eq:quasi-supermodular-cut} requires solving a lower-level problem to obtain $\hat{y}_{1}$ and then $n_{x}+1$ lower-level problems for $\varphi(\cdot,\hat{y}_1)$.
                Once we have obtained a new $z$ and thus $\hat{y}_{1}$, we need to recalculate all $\varphi(\cdot,\hat{y}_1)$.
                Although we can also set a pool to save calculated $\phi$ and $\varphi$, we have fewer repeated $\varphi$-values due to different $\hat{y}_{1}$-inputs.
                Hence, \eqref{eq:quasi-submodular-cut} or \eqref{eq:quasi-supermodular-cut} induces higher computational burden than \eqref{eq:submodular-cut} or \eqref{eq:supermodular-cut}, respectively.
                On the other hand, because we solve the optimal $\hat{y}_{1}$ and the right-hand side of \eqref{eq:quasi-submodular-cut} or \eqref{eq:quasi-supermodular-cut} constructs a facet-defining cut of $\varphi$, \eqref{eq:quasi-submodular-cut} or \eqref{eq:quasi-supermodular-cut} also shows high strengths.

    \section{Decision Rule-based Valid Inequalities}\label{sec:DR}
        In this section, we move to cases where special properties no longer exist.
        We leverage decision rules to solve the lower-level problem approximately and derive additional valid inequalities.
        We focus on MILP lower-level problems and recall the value function \eqref{eq:value function} as
        \begin{equation}
            \begin{aligned}
                \phi(x)=\max_{y_1\in Y_{1},y_{2}\in Y_{2}}~&d_{l1}^{\top}y_{1}+d_{l2}^{\top}y_{2}\\
                \text{s.t.}~&B_{l1}y_{1}+B_{l2}y_{2}\leq h_{\ell}-A_{\ell}x,
            \end{aligned}
        \end{equation}
        where $y_{1}\in Y_{1}\subseteq\{0,1\}^{n_{y_{1}}}$ and $y_{2}\in Y_{2}\subseteq\mathbb{R}^{n_{y_{2}}}$.
        In the following, we investigate linear decision rules in Section \ref{sec:DR1} and nonlinear ones in Section \ref{sec:DR2}.

        \subsection{Linear Decision Rule-based Valid Inequalities}\label{sec:DR1}
            We consider using linear decision rules to approximate the optimal policy.
            In particular, we apply linear decision rules to model the mapping from $x$ to the optimal $y_{2}$.
            To make the problem well-defined, we make the following assumption.
            \begin{assumption}\label{asp:feasible}
                For any $x\in\{0,1\}^{n_{x}}$, the lower-level problem is feasible and bounded.
            \end{assumption}

            The assumption is mild and commonly adopted in the literature \cite{KabirifarFotuhi-Firuzabad-3072,li2022public,qi2021sequential,smith2020survey}.
            Then, we can derive a valid inequality as presented in the following proposition.
            \begin{proposition}\label{pps:LDR-cut}
                For any $(\hat{\alpha},\hat{\beta})\in\Omega\subseteq \mathbb{R}^{n_{x}}\times\mathbb{R}$, we have a valid inequality for \eqref{eq:optimality} as
                \begin{equation}\label{eq:LDR-cut}
                    d_{\ell}^{\top}y\geq\hat{\alpha}^{\top}x+\hat{\beta}.
                \end{equation}
                Here, $\Omega$ is the projection onto $(\alpha,\beta)$ of the following set:
                \begin{equation}
                    \left\{\begin{aligned}
                    &\beta\leq d_{l1}^{\top}y_{1}+d_{l2}^{\top}y_{2}+\mathbf{1}^{\top}(U^{\top}d_{l2}-\alpha)^{-}\\
                    &B_{l1}y_{1}+B_{l2}y_{2}\leq h_{\ell}-(A_{\ell}+B_{l2}U)^{+}\mathbf{1}\\
                    &y_1\in\{0,1\}^{n_{y_{1}}},y_{2}\in\mathbb{R}^{n_{y_{2}}},U\in\mathbb{R}^{n_{y_{2}}\times n_{x}},\alpha\in\mathbb{R}^{n_{x}},\beta\in\mathbb{R}
                    \end{aligned}\right..
                \end{equation}
            \end{proposition}
            \begin{proof}\color{black}
                We seek a cut $\phi(x)\geq\alpha^{\top}x+\beta$ with coefficients $(\alpha,\beta)$ that is valid for all $x\in\{0,1\}^{n_{x}}$.
                Hence, $\forall x\in\{0,1\}^{n_{x}}$, we have
                \begin{equation*}
                    \alpha^{\top}x+\beta\leq\max
                    _{\substack{
                    y_1\in\{0,1\}^{n_{y_{1}}},y_{2}\in\mathbb{R}^{n_{y_{2}}}\\
                    B_{l1}y_{1}+B_{l2}y_{2}\leq h_{\ell}-A_{\ell}x}
                    }d_{l1}^{\top}y_{1}+d_{l2}^{\top}y_{2},
                \end{equation*}
                which implies
                \begin{equation*}
                    \beta\leq\min_{x\in\{0,1\}^{n_{x}}}\left\{
                    -\alpha^{\top}x+\max_{\substack{
                    y_1\in\{0,1\}^{n_{y_{1}}},y_{2}\in\mathbb{R}^{n_{y_{2}}}\\
                    B_{l1}y_{1}+B_{l2}y_{2}\leq h_{\ell}-A_{\ell}x}
                    }d_{l1}^{\top}y_{1}+d_{l2}^{\top}y_{2}
                    \right\}.
                \end{equation*}
                Note that $(\alpha,\beta)$ is valid if $\beta$ is no larger than a lower bound of the right-hand side.
                To this end, we lower bound the right-hand side as follows:
                \begin{equation*}
                    \begin{aligned}
                    &\min_{x\in\{0,1\}^{n_{x}}}\left\{
                    -\alpha^{\top}x+\max_{\substack{
                    y_1\in\{0,1\}^{n_{y_{1}}},y_{2}\in\mathbb{R}^{n_{y_{2}}}\\
                    B_{l1}y_{1}+B_{l2}y_{2}\leq h_{\ell}-A_{\ell}x}
                    }d_{l1}^{\top}y_{1}+d_{l2}^{\top}y_{2}
                    \right\}\\
                    =&\min_{x\in\{0,1\}^{n_{x}}}
                    \max_{y_1\in\{0,1\}^{n_{y_{1}}}}\left\{
                    -\alpha^{\top}x+d_{l1}^{\top}y_{1}+\max_{\substack{
                    y_{2}\in\mathbb{R}^{n_{y_{2}}}\\
                    B_{l2}y_{2}\leq h_{\ell}-A_{\ell}x-B_{l1}y_{1}}
                    }d_{l2}^{\top}y_{2}
                    \right\}\\
                    \geq&\max_{y_1\in\{0,1\}^{n_{y_{1}}}}
                    \min_{x\in\{0,1\}^{n_{x}}}\left\{
                    -\alpha^{\top}x+d_{l1}^{\top}y_{1}+\max_{\substack{
                    y_{2}\in\mathbb{R}^{n_{y_{2}}}\\
                    B_{l2}y_{2}\leq h_{\ell}-A_{\ell}x-B_{l1}y_{1}}
                    }d_{l2}^{\top}y_{2}
                    \right\}\\
                    \geq&\max_{y_1\in\{0,1\}^{n_{y_{1}}}}
                    \max_{\substack{
                    U\in\mathbb{R}^{n_{y_{2}}\times n_{x}},v\in\mathbb{R}^{n_{y_{2}}}\\
                    B_{l2}(Ux+v)\leq h_{\ell}-A_{\ell}x-B_{l1}y_{1},\forall x\in\{0,1\}^{n_{x}}}
                    }
                    \min_{x\in\{0,1\}^{n_{x}}}\left\{
                    -\alpha^{\top}x+d_{l1}^{\top}y_{1}+d_{l2}^{\top}(Ux+v)
                    \right\}\\
                    =&\max_{\substack{
                    y_1\in\{0,1\}^{n_{y_{1}}},U\in\mathbb{R}^{n_{y_{2}}\times n_{x}},y_{2}\in\mathbb{R}^{n_{y_{2}}}\\
                    B_{l1}y_{1}+B_{l2}y_{2}\leq h_{\ell}-\max_{x\in\{0,1\}^{n_{x}}}(A_{\ell}+B_{l2}U)x}
                    }
                    \left\{d_{l1}^{\top}y_{1}+d_{l2}^{\top}y_{2}+\min_{x\in\{0,1\}^{n_{x}}}
                    (U^{\top}d_{l2}-\alpha)^{\top}x
                    \right\}\\
                    =&\max_{\substack{
                    y_1\in\{0,1\}^{n_{y_{1}}},y_{2}\in\mathbb{R}^{n_{y_{2}}},U\in\mathbb{R}^{n_{y_{2}}\times n_{x}}\\
                    B_{l1}y_{1}+B_{l2}y_{2}\leq h_{\ell}-(A_{\ell}+B_{l2}U)^{+}\mathbf{1}}
                    }
                    d_{l1}^{\top}y_{1}+d_{l2}^{\top}y_{2}+\mathbf{1}^{\top}(U^{\top}d_{l2}-\alpha)^{-}.
                    \end{aligned}
                \end{equation*}
                Here, the first inequality follows the weak duality and the second inequality applies the linear decision rule $y_{2}=Ux+v$ with $U\in\mathbb{R}^{n_{y_{2}}\times n_{x}}$ and $v\in\mathbb{R}^{n_{y_{2}}}$.
                The subsequent equality replaces $v$ with $y_{2}$ and we note $y_{2}$ here has a different meaning from the original one.
                The last equality utilizes the binary-valued property of $x$.

                Therefore, $(\alpha,\beta)$ is valid if
                \begin{equation*}
                    \beta\leq\max_{\substack{
                    y_1\in\{0,1\}^{n_{y_{1}}},y_{2}\in\mathbb{R}^{n_{y_{2}}},U\in\mathbb{R}^{n_{y_{2}}\times n_{x}}\\
                    B_{l1}y_{1}+B_{l2}y_{2}\leq h_{\ell}-(A_{\ell}+B_{l2}U)^{+}\mathbf{1}}
                    }
                    d_{l1}^{\top}y_{1}+d_{l2}^{\top}y_{2}+\mathbf{1}^{\top}(U^{\top}d_{l2}-\alpha)^{-},
                \end{equation*}
                which implies that $(\alpha,\beta)$ is valid if there exists $(y_{1},y_{2},U)$ such that
                \begin{equation*}
                    \left\{\begin{aligned}
                    &\beta\leq d_{l1}^{\top}y_{1}+d_{l2}^{\top}y_{2}+\mathbf{1}^{\top}(U^{\top}d_{l2}-\alpha)^{-}\\
                    &B_{l1}y_{1}+B_{l2}y_{2}\leq h_{\ell}-(A_{\ell}+B_{l2}U)^{+}\mathbf{1}\\
                    &y_1\in\{0,1\}^{n_{y_{1}}},y_{2}\in\mathbb{R}^{n_{y_{2}}},U\in\mathbb{R}^{n_{y_{2}}\times n_{x}}
                    \end{aligned}\right..
                \end{equation*}
                We define the projection of this set onto $(\alpha,\beta)$ as $\Omega$ and end the proof.
            \end{proof}

            The separation of the valid inequality \eqref{eq:LDR-cut} can be done as follows:
            Given $(\hat{x},\hat{y})$, we solve
            \begin{equation}\label{eq:separation}
                \gamma^{*}:=\max_{(\alpha,\beta)\in\Omega}\hat{x}^{\top}\alpha+\beta
            \end{equation}
            and obtain its optimal solution $(\alpha^*,\beta^*)$.
            If $\gamma^*>d_{\ell}^{\top}\hat{y}$, we have $d_{\ell}^{\top}\hat{y}<(\alpha^*)^{\top}\hat{x}+\beta^*$, which means $(\hat{x},\hat{y})$ violates the above cut and should be removed.
            Then, we add a cut as
            \begin{equation}
                d_{\ell}^{\top}y\geq(\alpha^*)^{\top}x+\beta^*.
            \end{equation}

            However, due to the large $U$ and the binary-valued $y_{1}$, the separation problem \eqref{eq:separation} is a large-scale MILP and can be hard to solve.
            In practical implementation, we consider two methods to alleviate the computational burden.
            First, we fix some entries of $U$ to be zero and reduce the scale of \eqref{eq:separation}.
            This method keeps the validity of the generated inequality.
            Second, we fix $y_{1}$ in obtaining $\Omega$ and thus the separation problem \eqref{eq:separation} becomes a linear program.
            In light of the two methods, we have the following proposition:
            \begin{proposition}\label{pps:LDR-cut2}
                For any $\xi\in\{0,1\}^{n_{y_{2}}\times n_{x}}$, $\hat{y}_{1}\in Y_{1}$ and any $(\hat{\alpha},\hat{\beta})\in\Omega(\xi,\hat{y}_{1})\subseteq\mathbb{R}^{n_{x}}\times\mathbb{R}$, we have a valid inequality for \eqref{eq:optimality} as
                \begin{equation}\label{eq:LDR-cut2}
                    d_{\ell}^{\top}y\geq\hat{\alpha}^{\top}x+\hat{\beta}.
                \end{equation}
                Here, $\Omega(\xi,\hat{y}_{1})$ is the projection onto $(\alpha,\beta)$ of the following set:
                \begin{equation}
                    \left\{\begin{aligned}
                    &\beta\leq d_{l1}^{\top}\hat{y}_{1}+d_{l2}^{\top}y_{2}+\mathbf{1}^{\top}(U^{\top}d_{l2}-\alpha)^{-}\\
                    &B_{l1}\hat{y}_{1}+B_{l2}y_{2}\leq h_{\ell}-(A_{\ell}+B_{l2}U)^{+}\mathbf{1}\\
                    &U\odot\xi=0\\
                    &y_{2}\in\mathbb{R}^{n_{y_{2}}},U\in\mathbb{R}^{n_{y_{2}}\times n_{x}},\alpha\in\mathbb{R}^{n_{x}},\beta\in\mathbb{R}
                    \end{aligned}\right..
                \end{equation}
            \end{proposition}
            The proof of Proposition \ref{pps:LDR-cut2} is similar to that of Proposition \ref{pps:LDR-cut} and is thus omitted.

            In Proposition \ref{pps:LDR-cut2}, $\xi$ indicates the sparsity of $U$ and is selected according to specific problems.
            $\hat{y}_{1}$ is selected by solving the lower-level problem under the incumbent value of $x$.
            By replacing $\Omega$ in \eqref{eq:separation} with $\Omega(\xi,\hat{y}_{1})$, we have the separation problem and generate corresponding valid inequalities.

        \subsection{Trained Decision Rule-based Valid Inequalities}\label{sec:DR2}
            In many practical problems, we may repeatedly solve the BP model, especially the lower-level problem with respect to varying $x$ (or problem parameters).
            Motivated by this, we consider training a decision rule from historical data of the past solves to approximate the optimal policy at the lower level.
            Note that once an estimate of $y_{1}$ is available, the remaining problem becomes an LP and can be handled by existing methods.
            Hence, in this part, we only consider training decision rules to model the mapping from $x$ to the optimal $y_{1}$.
            Due to its strong power in data analysis and fitting, machine learning, such as neural networks, has drawn wide interest in function approximation \cite{zhoulearning, pmlr-v80-fujimoto18a, liang2017why, 1388456}.
            We utilize neural networks to approximate this mapping.
            Yet considering that learning is not the focus of the paper, we follow a general learning framework \cite{zhoulearning}, which is introduced in Appendix \ref{apd:learning}.

            We denote the trained decision rule as $\tilde{y}_{1}(x)$.
            Then, we derive a valid inequality as presented in the following proposition.
            \begin{proposition}\label{pps:learning-cut}
                For any fixed decision rule $\tilde{y}_{1}(x)$, we have a valid inequality for \eqref{eq:optimality} as
                \begin{equation}\label{eq:learning-cut}
                    d_{\ell}^{\top}y\geq d_{l1}^{\top}\tilde{y}_{1}(x)+\min_{\pi\in\mathbb{R}^{m_{\ell}}}\left\{
                        \big(h_{\ell}-A_{\ell}x-B_{l1}\tilde{y}_{1}(x)\big)^{\top}\pi
                        ~|~
                        \pi\geq0,B_{l2}^{\top}\pi=d_{l2}
                    \right\}.
                \end{equation}
            \end{proposition}
            \begin{proof}\color{black}
                We first rewrite \eqref{eq:optimality} as
                \begin{equation*}
                    d_{\ell}^{\top}y\geq\max_{y_{1}\in\{0,1\}^{n_{y_{1}}}}\left\{d_{l1}^{\top}y_{1}+\max_{\substack{y_{2}\in\mathbb{R}^{n_{y_{2}}}\\B_{l2}y_{2}\leq h_{\ell}-A_{\ell}x-B_{l1}y_{1}}}d_{l2}^{\top}y_{2}\right\}.
                \end{equation*}
                Incorporating the fixed decision rule $\tilde{y}_{1}(x)$, we have
                \begin{equation*}
                    d_{\ell}^{\top}y\geq d_{l1}^{\top}\tilde{y}_{1}(x)+\max_{\substack{y_{2}\in\mathbb{R}^{n_{y_{2}}}\\B_{l2}y_{2}\leq h_{\ell}-A_{\ell}x-B_{l1}\tilde{y}_{1}(x)}}d_{l2}^{\top}y_{2}.
                \end{equation*}
                Through dualization, we introduce dual variable $\pi$ and obtain \eqref{eq:learning-cut}.
                This completes the proof.
            \end{proof}

            The trained decision rule-based valid inequality \eqref{eq:learning-cut} is equivalent to
            \begin{subequations}\label{eq:learning-cut2}
                \begin{gather}
                    d_{\ell}^{\top}y\geq d_{l1}^{\top}+\big(h_{\ell}-A_{\ell}x-B_{l1}\tilde{y}_{1}\big)^{\top}\pi\\
                    \pi\geq0,B_{l2}^{\top}\pi=d_{l2}\\
                    \tilde{y}_{1}=\tilde{y}_{1}(x).
                \end{gather}
            \end{subequations}
            Bacause both $x$ and $\tilde{y}_{1}$ are binary-valued, the bilinear term $(h_{\ell}-A_{\ell}x-B_{l1}\tilde{y}_{1})^{\top}\pi$ can be linearized by the standard McCormick inequalities.
            The function $\tilde{y}_{1}(x)$ may be nonlinear and its linearization will be provided later.

            There are two approaches to utilizing \eqref{eq:learning-cut}.
            If we have high confidence in the quality of the trained decision rule $\tilde{y}_{1}(x)$, we can directly replace the optimality condition \eqref{eq:optimality} with \eqref{eq:learning-cut}.
            This approach quickly yields a lower bound with a solution $\hat{x}$, and by fixing $x=\hat{x}$ and solving \eqref{eq:VFR}, we can immediately obtain an upper bound and thus the corresponding optimality gap.
            Without confidence in the quality of $\tilde{y}_{1}(x)$, another approach is to add \eqref{eq:learning-cut} as an additional constraint and handle the optimality condition \eqref{eq:optimality} by Lagrangian-based valid inequalities.
            The latter approach is computationally heavier, but nevertheless provides an optimality guarantee.

            \subsubsection{Encoding Trained Decision Rules}
                Figure \ref{fig:NN} shows the adopted neural network architecture.
                There are $K$ hidden layers and one output layer in the architecture and we employ passthrough to enhance the representability.
                The input is $x$ and the output is $\tilde{y}_{1}$.
                Variable $z_{k}$ denotes the output of the $k$th hidden layer, $W_{k}/D_{k}$ and $b_{k}$ are the weights and biases of the $k$th layer, respectively, and $\sigma(\cdot)$ denotes the activation function.
                Specifically, the neural network defines $\tilde{y}_{1}(x)$ through
                \begin{subequations}\label{eq:NN}
                    \begin{align}
                        z_{1}&=\sigma(W_{1}x+b_{1}),\label{eq:NN1}\\
                        z_{k}&=\sigma(W_{k}z_{k-1}+b_{k}+D_{k}x), \quad \forall k=2,\ldots,K,\label{eq:NN2}\\
                        \tilde{y}_{1}&=\sigma(W_{K+1}z_{K}+b_{K+1}+D_{K+1}x).\label{eq:NN3}
                    \end{align}
                \end{subequations}
                \begin{figure}[!htbp]
                    \begin{center}
                        \vspace{-0ex}
                        \includegraphics[width=0.8\columnwidth]{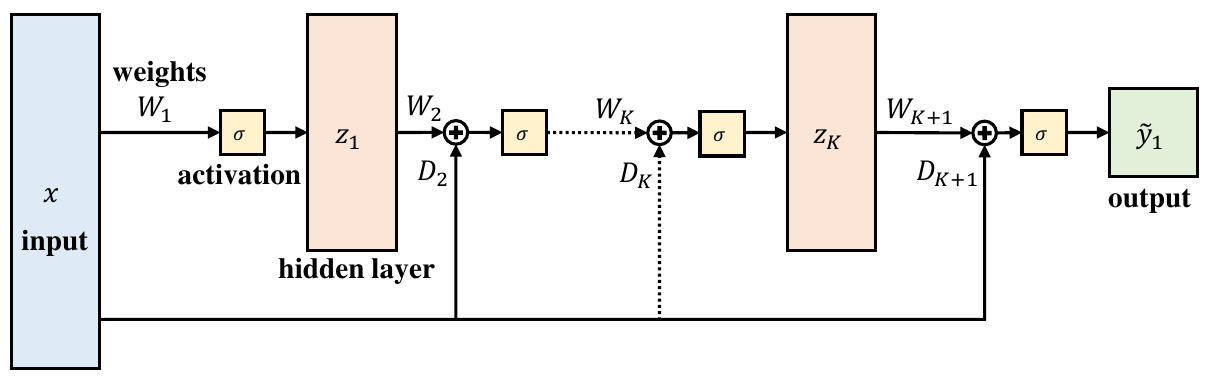}
                    \end{center}
                    \vspace{-2ex}
                    \caption{Neural network architecture.}
                    \vspace{-2ex}
                    \label{fig:NN}
                \end{figure}

                Considering that $y_{1}$ is binary-valued, we use the sigmoid function as activation, i.e., $\sigma(x):=1/(1+e^{-x})$.
                In light of the nonlinear $\sigma(x)$, we conduct the following linearization in encoding $\tilde{y}_{1}(x)$.
                For the sigmoid activation in the hidden layers, i.e., $\sigma(\cdot)$ in \eqref{eq:NN1} and \eqref{eq:NN2}, we approximate it by a piecewise linear function $\tilde{\sigma}_{p}(x):=\text{clip}(x/5+1/2,[0,1])$.
                In Appendix \ref{apd:sigmoid}, we compare the sigmoid function $\sigma(x)$ and the piecewise linearization $\tilde{\sigma}_{p}(x)$.
                Note that $\tilde{\sigma}_{p}(x)$ can be further linearized as
                \begin{subequations}
                    \begin{align}
                        -M\delta_{1}+\frac{1}{5}x+\frac{1}{2}\leq~&\tilde{\sigma}_{p}~\leq \frac{1}{5}x+\frac{1}{2}+M\delta_{0}\\
                        \delta_{1}\leq~&\tilde{\sigma}_{p}~\leq 1-\delta_{0},
                    \end{align}
                \end{subequations}
                where $\delta_{0}$ and $\delta_{1}$ are auxiliary binary variables and $M$ is a sufficiently large constant.
                For the sigmoid activation of the output layer, i.e., $\sigma(x)$ in \eqref{eq:NN3}, we approximate it by a unit step function $\tilde{\sigma}_{s}(x):=\text{step}(x)$, guaranteeing that $\tilde{y}_{1}$ is binary-valued.
                The step function $\tilde{\sigma}_{s}(x)$ can be further linearized as
                \begin{equation}
                    -M(1-\tilde{\sigma}_{s})\leq~x~\leq M\tilde{\sigma}_{s}.
                \end{equation}
                We note that the approximation from $\sigma(\cdot)$ to $\tilde{\sigma}_{p}(\cdot)$ or $\tilde{\sigma}_{s}(\cdot)$ influences the accuracy of the practically encoded decision rule $\tilde{y}_{1}(x)$, but does not break the validity of \eqref{eq:learning-cut}.

    \section{Numerical Results}\label{sec:results}
        In this section, we conduct numerical experiments in both general bilevel MILP instances and facility location interdiction problems, both of which contain binary tenders .
        We compare with the state-of-the-art solver for bilevel problems, \texttt{MibS} \cite{tahernejad2020branch}, to validate the effectiveness and performance of our methods.
        All experiments are implemented in Python 3.7 and solved by Gurobi 10.0.3 on a computer equipped with an Intel i7-10870H CPU and 16 GB of RAM.
        The time limit is one hour.

        \subsection{General Problems}
            We first consider general bilevel MILP instances as shown in \eqref{eq:bilevel}.
            For comparison, we follow the instance generation rules of \cite{tahernejad2020branch} and more details are placed in Appendix \ref{apd:instance-general}.
            We generate one small instance with $n_{x}=10$ and ten larger instances with the dimension $n_{x}$ of the leader's decision $x$ ranging from 200 to 2000.
            We compare the \texttt{MibS} solver and our Lagrangian-based methods (using the quick calculation of their coefficients) in these instances.

            \begin{table}[!b]
                    \small
                    \renewcommand{\arraystretch}{1.1}
                    \vspace{-0ex}
                    \caption{Comparison of solution quality in general problems}
                    \vspace{-2ex}
                    \label{tab:general-objective}
                    \centering
                    \begin{tabular}{ccccccc}
                        \toprule
                        \multirow{2}{*}{$n_x$} & \multicolumn{2}{c}{\texttt{MibS}} & \multicolumn{2}{c}{Gurobi - Penalty} & \multicolumn{2}{c}{Gurobi - Lagrangian} \\
                        \cmidrule(lr){2-3}\cmidrule(lr){4-5}\cmidrule(lr){6-7}
                                               & Objective    & gap       & Objective            & gap           & Objective             & gap             \\
                        \midrule
                        10                     & $-$72.56     & 0         & $-$72.56             & 0             & $-$72.56              & 0               \\
                        200                    & $-$170.53    & 0         & $-$175.93            & 0             & $-$175.93             & 0               \\
                        400                    & $-$91.02     & 0         & $-$106.61            & 0             & $-$106.61             & 0               \\
                        600                    & $-$105.62    & 0         & $-$105.62            & 0             & $-$105.62             & 0               \\
                        800                    & $-$118.78    & 0         & $-$118.78            & 0             & $-$118.78             & 0               \\
                        1000                   & $-$96.83     & 0         & $-$96.83             & 0             & $-$96.83              & 0               \\
                        1200                   & $-$64.01     & 7.61\%    & $-$66.67             & 0             & $-$66.67              & 0               \\
                        1400                   & $-$68.99     & 4.68\%    & $-$82.46             & 0             & $-$82.46              & 0               \\
                        1600                   & $-$89.50     & 0         & $-$89.50            & 0             &   $\backslash$                    &    $\backslash$             \\
                        1800                   & $-$57.46     & 12.36\%   & $-$68.73             & 0             &  $\backslash$                     &    $\backslash$             \\
                        2000                   & $-$71.02     & 24.93\%   & $-$76.38             & 0             & $\backslash$                      &$\backslash$\\
                        \bottomrule
                    \end{tabular}
                \end{table}

                \begin{table}[!b]
                    \small
                    \renewcommand{\arraystretch}{1.1}
                    \vspace{-0ex}
                    \caption{Comparison of solution time in general problems}
                    \vspace{-2ex}
                    \label{tab:general-time}
                    \centering
                    \begin{tabular}{cccccc}
                        \toprule
                        \multirow{2}{*}{$n_x$} & \texttt{MibS}              & \multicolumn{2}{c}{Gurobi - Penalty}                         & \multicolumn{2}{c}{Gurobi - Lagrangian}                     \\
                        \cmidrule(lr){2-2}\cmidrule(lr){3-4}\cmidrule(lr){5-6}
                                               & \begin{tabular}[c]{@{}c@{}}time for\\solution (s)\end{tabular} & \begin{tabular}[c]{@{}c@{}}time for\\solution (s)\end{tabular} & \begin{tabular}[c]{@{}c@{}}time for\\$\hat{\rho}$ (s)\end{tabular} & \begin{tabular}[c]{@{}c@{}}time for\\solution (s)\end{tabular} & \begin{tabular}[c]{@{}c@{}}time for\\ $U\&L$ (s)\end{tabular}\\
                        \midrule
                        10                     & 0.14              & 0.24              & 0.02                                     & 0.25              & 0.28                                    \\
                        200                    & 10.26             & 4.09              & 0.24                                     & 4.30              & 78.89                                   \\
                        400                    & 52.50             & 8.00              & 2.38                                     & 7.56              & 393.17                                  \\
                        600                    & 244.57            & 21.52             & 4.18                                     & 18.08             & 1081.78                                 \\
                        800                    & 579.31            & 27.55             & 13.84                                    & 26.73             & 2523.72                                 \\
                        1000                   & 811.37            & 41.98             & 29.25                                    & 35.58             & 4584.26                                 \\
                        1200                   & 3787.79           & 244.99            & 34.42                                    & 217.27            & 8055.63                                 \\
                        1400                   & 3848.30           & 238.77            & 89.56                                    & 204.85            & 13547.33                                \\
                        1600                   & 3618.36           & 198.46            & 144.43                                   &  $\backslash$                 &     $>$ 4 hours                                     \\
                        1800                   & 4030.51           & 577.60            & 373.65                                   &  $\backslash$                 &   $>$ 4 hours                                       \\
                        2000                   & 4209.03           & 815.86            & 302.62                                   &  $\backslash$                 &$>$ 4 hours \\
                        \bottomrule
                    \end{tabular}
            \end{table}

            \subsubsection{Performance Comparison}
                The numerical results are provided in Tables \ref{tab:general-objective} and \ref{tab:general-time}.
                In terms of the solution quality shown in Table \ref{tab:general-objective}, both the penalty-based and Lagrangian-based valid inequalities lead to optimal solutions in all instances with \(n_x \leq 1400\), while \texttt{MibS} cannot find optimal solutions when $n_{x}$ exceeds 1000. When \(n_x\) further increases beyond 1400, the penalty-based inequality is still able to prove global optimum, while both \texttt{MibS} and the Lagrangian-based inequality cannot do so.
                The comparison validates the superiority of the penalty-based inequality in finding optimal solutions for general instances.

                In terms of the solution time shown in Table \ref{tab:general-time}, when \(n_x \leq 1400\), the Lagrangian-based valid inequality yields the shortest solution time, followed by the penalty-based valid inequality.
                In addition, Algorithm~\ref{alg:bilevel} using either penalty-based or Lagrangian-based valid inequality outperforms \texttt{MibS} in solution time.
                This validates the computational efficiency of the proposed approaches.
                Nevertheless, the time for obtaining the Lagrangian coefficients $\hat{\rho}$ and $U/L$ increases as $n_{x}$ increases and becomes non-trivial when $n_{x}$ is large. In particular, when $n_{x}$ exceeds 1400, the time for obtaining Lagrangian coefficients $U/L$ is longer than 4 hours, rendering the Lagrangian-based inequality inefficient for large instances.
                In contrast, the time for obtaining coefficient \(\hat{\rho}\) for the penalty-based valid inequality as well as the corresponding branch-and-cut algorithm are still significantly shorter than that of \texttt{MibS} even in these large instances. This demonstrates the good scalability of applying the penalty-based valid inequality.

                With the above results and analysis, we demonstrate that our Lagrangian-based methods show better performance and scalability than the \texttt{MibS} solver in general instances.
                In particular, we suggest adopting the penalty-based valid inequality when the dimension \(n_x\) of the tender variables is high.

            \subsubsection{Sensitivity Analysis}
                We further validate the performance of our proposed methods under different conditions.
                We use the $n_{x}=2000$ instance and adopt the penalty-based valid inequality for sensitivity analysis.

                We first conduct a sensitivity analysis on the ratio of binary variables in $y$ and raise the ratio from $10\%$ to $90\%$.
                The results are reported in Figure \ref{fig:general-binary}.
                We see that our method can always find optimal solutions.
                As the ratio of binary variables in $y$ increases, which means the lower-level problem transforms from a pure linear program to a pure integer program, the solution time of our method decreases.
                However, when the ratio reaches $90\%$, the time for obtaining the penalty coefficient gets much longer, which may influence the overall efficiency.

                \begin{figure}[!b]
                    \begin{center}
                        \vspace{-1ex}
                        \includegraphics[width=0.65\columnwidth]{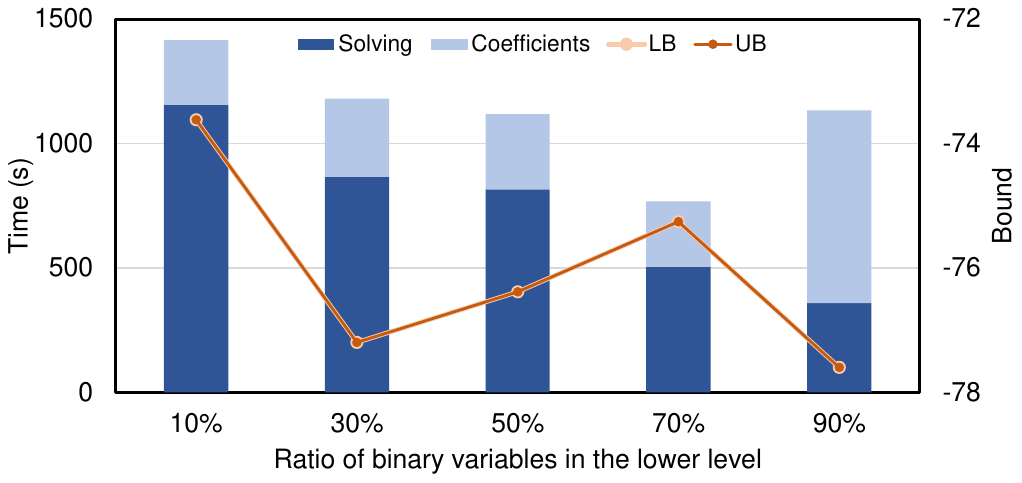}
                    \end{center}
                    \vspace{-4ex}
                    \caption{Sensitivity analysis on the ratio of binary variables in $y$.}
                    \vspace{-2ex}
                    \label{fig:general-binary}
                \end{figure}

                We then conduct a sensitivity analysis on the scale of the lower-level problem and increase $n_y$ from $0.5n_{x}$ to $1.5n_{x}$.
                The number of lower-level constraints also increases according to the instance generation rule (see Appendix \ref{apd:instance-general}).
                The results are reported in Figure \ref{fig:general-scale}.
                We see that our method can always find optimal solutions.
                As the scale of the lower-level problem increases, both the solution time and the time for obtaining penalty coefficients increase quickly.

                We finally conduct a sensitivity analysis on the sparsity of lower-level constraint coefficients and increase the sparsity from $0\%$ to $50\%$.
                When the sparsity is higher than $50\%$, we cannot find a feasible solution within one hour.
                The results are reported in Figure \ref{fig:general-sparsity}.
                We see that when the sparsity is equal to or lower than $30\%$, our method can find optimal solutions and the computational time is comparable.
                But when the sparsity is equal to or higher than $40\%$, the computational burden gets much heavier and we cannot find optimal solutions within 1 hour.

                \begin{figure}[!t]
                    \begin{center}
                        \vspace{-0ex}
                        \includegraphics[width=0.65\columnwidth]{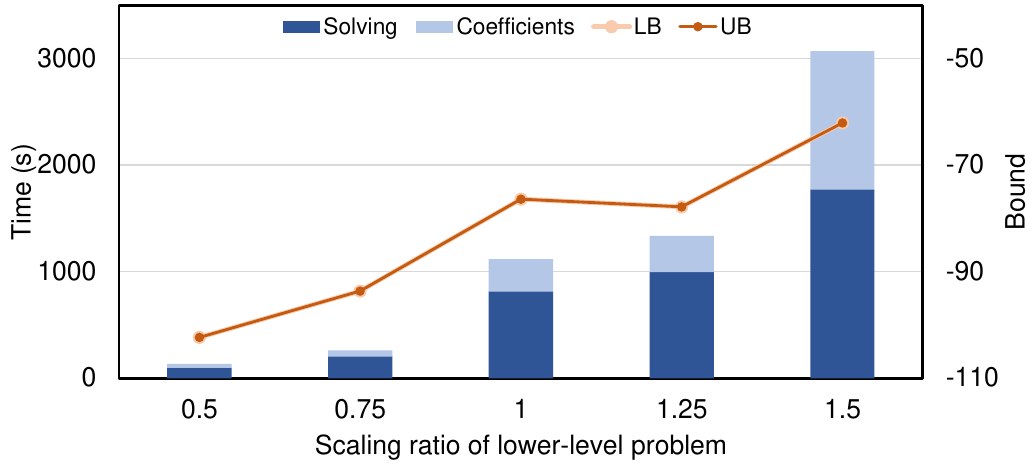}
                    \end{center}
                    \vspace{-4ex}
                    \caption{Sensitivity analysis on the scaling ratio of the lower-level problem.}
                    \vspace{-0ex}
                    \label{fig:general-scale}
                \end{figure}

                \begin{figure}[!t]
                    \begin{center}
                        \vspace{-0ex}
                        \includegraphics[width=0.65\columnwidth]{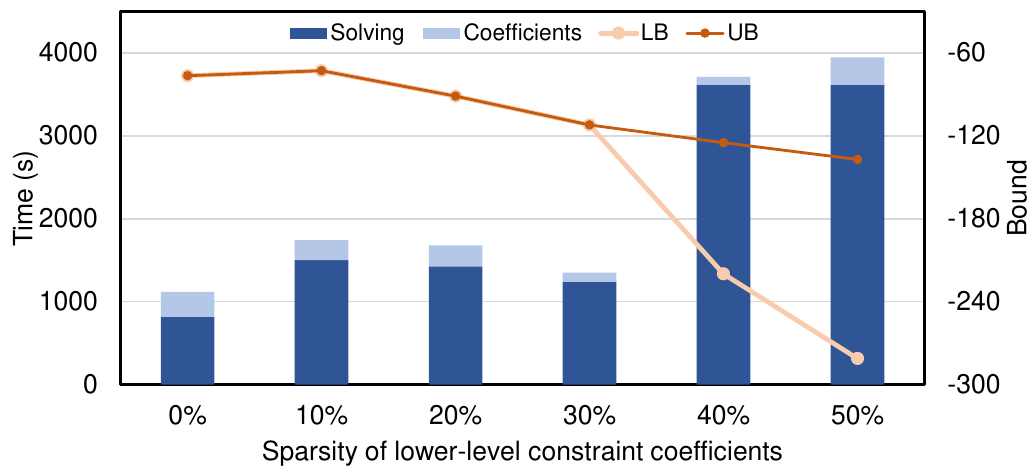}
                    \end{center}
                    \vspace{-4ex}
                    \caption{Sensitivity analysis on the sparsity of lower-level constraint coefficients.}
                    \vspace{-2ex}
                    \label{fig:general-sparsity}
                \end{figure}

        \subsection{Performance in Facility Location Interdiction Problems}
            We consider the facility location interdiction problem, whose lower-level problem possesses a special formulation.
            With this special problem structure, we compare the \texttt{MibS} solver with our Lagrangian-based, special property-based, and decision rule-based methods.

            \subsubsection{Problem Settings}
                Consider the following facility location interdiction problem:
                \begin{align*}
                    \min_{x,y_{1},y_{2}}&~-g(y_{1},y_{2})\\
                    \textrm{s.t.}&~1^{\top}(1-x)\leq B, x\in\{0,1\}^{n}\\
                    &~(y_{1},y_{2})\in\arg\max_{(y_{1},y_{2})\in Y(x)}-g(y_{1},y_{2})\\  
                    &~\text{with} \ \ Y(x)=\left\{
                        (y_{1},y_{2})\left|\begin{aligned}
                            &\sum_{i\in[n]}y_{2,ij}\leq1,\forall j\in[m]\\
                            &\sum_{j\in[m]}d_{j}y_{2,ij}\leq C_{i}x_{i}+C'_{i}y_{1,i},\forall i\in[n]\\
                            &y_{1}\in\{0,1\}^{n}, y_{2}\in\mathbb{R}_{+}^{n\times m}
                        \end{aligned}\right.
                    \right\},
                \end{align*}
                where $n$ is the number of facilities and $m$ is the number of customers.
                The upper level (attacker) decides the interdiction action $x$ to maximize the operation cost $g(\cdot)$, where $x_{i}=0$ indicates that facility $i$ is attacked.
                $B$ is the budget that restricts the maximum number of attacked facilities.
                The lower level (facility operator) decides the repair action $y_{1}$ and the transportation flow $y_{2}$ to minimize the operation cost $g(\cdot)$, where $y_{1,i}=1$ indicates that facility $i$ is repaired and $y_{2,ij}$ indicates the flow from facility $i$ to customer $j$.
                $d_{j}$ is the demand of customer $j$.
                $C_{i}$ is the original capacity of facility $i$ and $C'_{i}$ is the additional capacity through repair.
                The operational cost is given as
                \begin{align*}
                    c_{r}^{\top}y_{1}+\sum_{i\in[n]}\sum_{j\in[m]}c_{t,ij}d_{j}y_{2,ij}+\sum_{j\in[m]}c_{p}d_{j}\left(1-\sum_{i\in[n]}y_{2,ij}\right),
                \end{align*}
                where $c_{r,i}$ is the repair cost of facility $i$, $c_{t, ij}$ is the unit transportation cost from facility $i$ to customer $j$, $c_{p}$ is the penalty coefficient of unmet demands.
                We ignore constant terms and write $g(y_{1},y_{2})$ as
                \begin{align*}
                    g(y_{1},y_{2}):=c_{r}^{\top}y_{1}+\sum_{i\in[n]}\sum_{j\in[m]}c'_{t,ij}d_{j}y_{2,ij},
                \end{align*}
                where $c'_{t,ij}:=c_{t,ij}-c_{p}<0$.
                Following the instance generation rule in Appendix \ref{apd:instance-facility}, we generate 6 instances with $n=5, 10, 15, 20, 25, 30$ for numerical experiments. Notably, the value function of the above facility location interdiction problem admits quasi-submodularity.
                \begin{proposition}[\cite{kanglinRFLP}]\label{pps:CFLI-property}
                    Given any $y_{1}\in\{0,1\}^{n}$, we have
                    \begin{align*}
                        \varphi(x,y_{1}):=\max_{(y_{1},y_{2})\in Y(x)}-g(y_{1},y_{2})
                    \end{align*}
                    is submodular in $x\in\{0,1\}^{n}$.
                \end{proposition}
                Proposition \ref{pps:CFLI-property} identifies the submodularity of the lower-level problem for any fixed $y_{1}$ and thus the quasi-submodularity of the lower-level problem.

            \subsubsection{Performance Enhancement through Submodularity}
                We first ignore the repair of facilities in the lower level, i.e., fix $y_{1}=0$, and hence, the lower-level problem becomes a linear program.
                According to Proposition \ref{pps:CFLI-property}, the lower-level value function is submodular in $x\in\{0,1\}^{n}$, which enables both the efficient and exact calculation of Lagrangian coefficients and the submodularity-based valid inequalities.
                In this part, we compare the \texttt{MibS} solver with our submodularity-based method and Lagrangian-based method (exact coefficients).
                The numerical results are provided in Tables \ref{tab:special-MIBS} and \ref{tab:special-Lagrangian}.

                \begin{table}[!htbp]
                    \small
                    \renewcommand{\arraystretch}{1.1}
                    \vspace{-0ex}
                    \caption{Performance comparison with a LP lower-level problem}
                    \vspace{-2ex}
                    \label{tab:special-MIBS}
                    \centering
                    \begin{tabular}{ccccc}
                        \toprule
                        \multirow{2}{*}{$n$} & \multicolumn{4}{c}{\texttt{MibS}}                                                       \\
                        \cmidrule(lr){2-5}
                         & objective & bound  & gap      & \begin{tabular}[c]{@{}c@{}}time   for\\      solution (s)\end{tabular} \\
                        \midrule
                        5                     & 130.99    & 130.99 & 0.00\%   & 1.30                                                                  \\
                        10                    & 321.79    & 321.79 & 0.00\%   & 4.31                                                                  \\
                        15                    & 554.20    & 554.20 & 0.00\%   & 174.67                                                                \\
                        20                    & 736.72    & 0.00   & 100.00\% & 3600.00                                                               \\
                        25                    & 1100.70   & 0.00   & 100.00\% & 3678.23                                                               \\
                        30                    & 1139.10   & 0.00   & 100.00\% & 3738.96 \\
                        \bottomrule
                    \end{tabular}
                \end{table}

                \begin{table}[!htbp]
                    \small
                    \renewcommand{\arraystretch}{1.1}
                    \vspace{-0ex}
                    \caption{Performance comparison with a LP lower-level problem (continued)}
                    \vspace{-2ex}
                    \label{tab:special-Lagrangian}
                    \centering
                    \newcommand{\tabincell}[2]{\begin{tabular}{@{}#1@{}}#2\end{tabular}}
                    \begin{tabular}{cccccccccc}
                        \toprule
                         \multirow{2}{*}{$n$} & \multicolumn{4}{c}{{\color[HTML]{333333} Gurobi - Submodularity}}                                     & \multicolumn{5}{c}{{\color[HTML]{333333} Gurobi - Lagrangian (exact coefficients)}}                                                                                        \\
                        \cmidrule(lr){2-5}\cmidrule(lr){6-10}
                         & objective & bound   & gap     & \begin{tabular}[c]{@{}c@{}}time   for\\      solution (s)\end{tabular} & objective & bound  & gap     & \begin{tabular}[c]{@{}c@{}}time   for\\      solution (s)\end{tabular} & \begin{tabular}[c]{@{}c@{}}time   for\\      $U\&L$ (s)\end{tabular} \\
                        \midrule
                        5                     & 130.99    & 130.99  & 0.00\%  & 0.28                                                                  & 130.99    & 130.99 & 0.00\%  & 0.16                                                                  & 0.18                                                                \\
                        10                    & 321.79    & 321.79  & 0.00\%  & 1.87                                                                  & 321.79    & 321.79 & 0.00\%  & 0.92                                                                  & 1.09                                                                \\
                        15                    & 554.20    & 554.20  & 0.00\%  & 91.62                                                                 & 554.20    & 554.20 & 0.00\%  & 38.38                                                                 & 3.69                                                                \\
                        20                    & 736.72    & 736.72  & 0.00\%  & 1243.32                                                               & 736.72    & 736.72 & 0.00\%  & 484.04                                                                & 8.26                                                                \\
                        25                    & 1130.08   & 913.22  & 19.19\% & 3642.34                                                               & 1094.46   & 946.42 & 13.53\% & 3603.05                                                               & 17.61                                                               \\
                        30                    & 1125.29   & 1035.04 & 8.02\%  & 3601.14                                                               & 1129.72   & 940.38 & 16.76\% & 3603.66                                                               & 32.43 \\
                        \bottomrule
                    \end{tabular}
                \end{table}

                From Tables \ref{tab:special-MIBS} and \ref{tab:special-Lagrangian}, we can see that when $n\leq15$, all three methods can find optimal solutions, but both the Lagrangian-based and submodularity-based methods show higher computational efficiency.
                When $n=20$, \texttt{MibS} is unable to find a good lower bound within one hour and thus gives a 100$\%$ gap, while the Lagrangian-based and submodularity-based methods is able to find the optimal solution in about 20 minutes and 8 minutes, respectively.
                Comparing the Lagrangian-based and submodularity-based methods, when $n\leq20$, the Lagrangian-based method shows relatively higher computation efficiency.
                When $n\geq25$, all three methods can not find optimal solutions.
                However, \texttt{MibS} provides a much weaker bound and thus a large optimization gap, while the Lagrangian-based and submodularity-based methods prove significantly smaller optimality gaps.
                The above results and analysis validate the effectiveness of our methods and the benefit of incorporating submodularity.

            \begin{table}[!b]
                    \small
                    \renewcommand{\arraystretch}{1.1}
                    \vspace{-0ex}
                    \caption{Performance comparison with a MILP lower-level problem}
                    \vspace{-2ex}
                    \label{tab:quasi-MIBS}
                    \centering
                        \begin{tabular}{cccccccccc}
                        \toprule
                        \multirow{2}{*}{$n$}                    & \multicolumn{4}{c}{\texttt{MibS}}                                                        & \multicolumn{5}{c}{{\color[HTML]{333333} Gurobi - Lagrangian (relaxed coefficients)}}                                                                                        \\
                        \cmidrule(lr){2-5}\cmidrule(lr){6-10}
                         & objective & bound   & gap      & \begin{tabular}[c]{@{}c@{}}time   for\\      solution (s)\end{tabular} & objective & bound   & gap      & \begin{tabular}[c]{@{}c@{}}time   for\\      solution (s)\end{tabular} & \begin{tabular}[c]{@{}c@{}}time   for\\      $U\&L$ (s)\end{tabular} \\
                        \midrule
                        5                     & 205.95    & 205.95  & 0.00\%   & 0.68                                                                  & 205.95    & 205.95  & 0.00\%   & 0.34                                                                  & 0.39                                                                \\
                        10                    & 451.49    & 451.49  & 0.00\%   & 23.32                                                                 & 451.49    & 451.49  & 0.00\%   & 14.82                                                                 & 2.61                                                                \\
                        15                    & 682.21    & 682.21  & 0.00\%   & 1884.33                                                               & 682.21    & 682.21  & 0.00\%   & 1390.63                                                               & 9.38                                                                \\
                        20                    & 919.71    & $-$97.61  & 110.61\% & 3629.90                                                               & 928.27    & $-$97.61  & 110.52\% & 3637.22                                                               & 21.77                                                               \\
                        25                    & 1210.92   & $-$115.75 & 109.56\% & 3668.27                                                               & 1229.08   & $-$115.75 & 109.42\% & 3604.59                                                               & 47.88                                                               \\
                        30                    & 1438.27   & $-$129.90 & 109.03\% & 3737.12                                                               & 1461.22   & $-$129.90 & 108.89\% & 3730.72                                                               & 93.24 \\
                        \bottomrule
                    \end{tabular}
                \end{table}

                \begin{table}[!htbp]
                    \small
                    \renewcommand{\arraystretch}{1.1}
                    \vspace{-0ex}
                    \caption{Performance comparison with a MILP lower-level problem (continued)}
                    \vspace{-2ex}
                    \label{tab:quasi-EPI}
                    \centering
                    \newcommand{\tabincell}[2]{\begin{tabular}{@{}#1@{}}#2\end{tabular}}
                    \begin{tabular}{ccccccccc}
                        \toprule
                        \multirow{2}{*}{$n$} & \multicolumn{4}{c}{Gurobi - Quasi-submodularity}                                                                                         & \multicolumn{4}{c}{Gurobi - Lagrangian (exact coefficients)}                                                                          \\
                        \cmidrule(lr){2-5}\cmidrule(lr){6-9}
                                             & objective        & bound      & gap              & \begin{tabular}[c]{@{}c@{}}time   for\\      solution (s)\end{tabular} & objective        & bound      & gap              & \begin{tabular}[c]{@{}c@{}}time   for\\      solution (s)\end{tabular} \\
                        \midrule
                        5                    & 205.95         & 205.95   & 0                & 1.37                                                                  & 205.95         & 205.95   & 0                & 2.53                                                                  \\
                        10                   & 451.49          & 451.49    & 0                & 88.59                                                                 & 451.49          & 451.49    & 0                & 147.67                                                                \\
                        15                   & 684.26         & 638.96 & 6.62\%           & 3606.42                                                              & 685.75         & 636.85 & 7.13\%           & 3601.58                                                               \\
                        20                   & 934.15         & 752.55 & 19.44\%          & 3633.39                                                             & 941.47         & 719.94 & 23.53\%          & 3611.91                                                               \\
                        25                   & 1223.46         & 918.82 & 24.90\%          & 3700.69                                                               & 1227.29        & 901.94 & 26.51\%          & 3651.38                                                               \\
                        30                   & 1440.19 & 1078.66  & 25.10\% & 3626.86                                                             & 1436.58 & 1059.96  & 26.22\% & 3617.68 \\
                        \bottomrule
                    \end{tabular}
                \end{table}

            \subsubsection{Performance Enhancement through Quasi-Submodularity}
                We consider the facility location interdiction problem without fixing $y_{1}$, and hence, the lower-level problem is a MILP.
                According to Proposition \ref{pps:CFLI-property} and Definition \ref{def:quasi}, the lower-level value function is quasi-submodular in $x\in\{0,1\}^{n}$, which enables both the quasi-submodularity based valid inequality and the exact calculation of Lagrangian coefficients with fixed $y_{1}$.
                In this part, we compare the \texttt{MibS} solver with our Lagrangian-based method (relaxed coefficients, not using quasi-submodularity), quasi-submodularity based method, and Lagrangian-based method (exact coefficients, using quasi-submodularity).
                The numerical results are provided in Tables \ref{tab:quasi-MIBS} and \ref{tab:quasi-EPI}.

                When we do not utilize quasi-submodularity, the results are provided in Tables \ref{tab:quasi-MIBS}.
                We can see that when $n\leq15$, both \texttt{MibS} and the Lagrangian-based method can find optimal solutions, while the Lagrangian-based method shows higher computational efficiency.
                When $n\geq20$, both methods can not find optimal solutions within one hour, and the Lagrangian-based method proves a slightly smaller optimality gap.

                In contrast, when we utilize quasi-submodularity, the results are provided in Tables \ref{tab:quasi-EPI}.
                We can see that when $n\leq10$, both the quasi-submodularity based method and the exact Lagrangian-based method consume more time in finding optimal solutions than \texttt{MibS} or the relaxed Lagrangian-based method, and for $n \geq 15$, they cannot even find optimal solutions within one hour.
                This results from the longer time in generating valid inequalities that depend on the incumbent $\hat{y}_{1}$ and, more specifically, a heavier computational burden in calculating $\varphi(x,\hat{y}_{1})$.
                However, when $n\geq20$, the quasi-submodularity based method and the exact Lagrangian-based method prove a significantly smaller optimality gap than that of \texttt{MibS} or the relaxed Lagrangian-based method.
                This demonstrates the strength of these valid inequalities, which significantly improves the lower bound.
                In particular, the quasi-submodularity based method proves a slightly smaller optimality gap than the exact Lagrangian-based method.
                The above results and analysis validate the superiority of our methods and the benefit of incorporating quasi-submodularity in larger-scale instances.

            \begin{table}[!b]
                    \small
                    \renewcommand{\arraystretch}{1.1}
                    \vspace{-0ex}
                    \caption{Performance with additional LDR-based valid inequalities}
                    \vspace{-2ex}
                    \label{tab:quasi-LDR}
                    \centering
                    \begin{tabular}{cccccc}
                        \toprule
                        \multirow{2}{*}{$n$}                      & \multicolumn{5}{c}{Gurobi - Lagrangian (relaxed coefficients) + LDR}                                                                        \\
                        \cmidrule(lr){2-6}
                         & objective & bound   & gap     & \begin{tabular}[c]{@{}c@{}}time   for\\      solutions (s)\end{tabular} & \begin{tabular}[c]{@{}c@{}}time for\\      $U/L$ (s)\end{tabular} \\
                        \midrule
                        5                     & 205.95    & 205.95  & 0  & 2.84                                                                  & 0.39                                                                        \\
                        10                    & 451.49    & 451.49  & 0  & 115.62                                                                & 2.55                                                                        \\
                        15                    & 685.42    & 635.74  & 7.25\%  & 3665.29                                                               & 8.63                                                                        \\
                        20                    & 965.99    & 731.84  & 24.24\% & 3603.08                                                               & 23.31                                                                       \\
                        25                    & 1224.82   & 903.65  & 26.22\% & 3620.37                                                               & 47.92                                                                       \\
                        30                    & 1520.38   & 1074.45 & 29.33\% & 3631.24                                                               & 100.79 \\
                        \bottomrule
                    \end{tabular}
                \end{table}

            \subsubsection{Performance Enhancement through Linear Decision Rules}
                The LDR-based valid inequality does not rely on special properties and holds in general for facility location interdiction problems.
                In this part, we ignore the quasi-submodularity of the lower-level problem and consider the LDR-based valid inequalities, in addition to the Lagrangian-based valid inequality with relaxed coefficients.
                The results are shown in Table \ref{tab:quasi-LDR} and are compared to Table \ref{tab:quasi-MIBS}.

                We can see that, with additional LDR-based valid inequalities, the solution time gets longer when $n\leq10$ and we cannot get the optimal solution within one hour when $n=15$.
                The results are reasonable because additional computation is required to obtain LDR-based valid inequalities.
                When $n\geq20$, we cannot get optimal solutions within one hour, either with or without LDR-based valid inequalities.
                However, with LDR-based valid inequalities, the lower bound gets significantly improved, and accordingly the gap is reduced significantly.
                The above results and analysis validate the effectiveness of LDR-based valid inequalities in improving the lower bound in large-scale general instances.

            \subsubsection{Performance Enhancement Through Trained Decision Rules}
                Finally, we 
                demonstrate the trained decision rule-based valid inequality.
                By default, we use 1,000 samples and adopt the Adam algorithm \cite{kingma2014adam} for training.
                We set the training epochs to be 1,000 and the learning rate to be 0.01 with a 0.001 decay.
                We consider two approaches:
                i) replacing the optimality condition~\eqref{eq:optimality} with the trained decision rule-based valid inequality, and
                ii) add the trained decision rule-based valid inequality based on the relaxed Lagrangian-based method.
                The results are reported in Table \ref{tab:quasi-learning} and are compared to Table \ref{tab:quasi-MIBS}.

            \begin{table}[!htbp]
                    \small
                    \renewcommand{\arraystretch}{1.1}
                    \vspace{-0ex}
                    \caption{Results of the proposed learning-based valid inequality}
                    \vspace{-2ex}
                    \label{tab:quasi-learning}
                    \centering
                    \begin{tabular}{ccccccccc}
                        \toprule
                                              & \multicolumn{4}{c}{{\color[HTML]{000000} Gurobi - Trained decision rule}}                       & \multicolumn{4}{c}{+ Lagrangian (relaxed coefficients)}                                              \\
                        \cmidrule(lr){2-5}\cmidrule(lr){6-9}
                        \multirow{-2}{*}{$n$} & objective & bound   & gap     & \begin{tabular}[c]{@{}c@{}}time for\\      solutions (s)\end{tabular} & objective & bound   & gap     & \begin{tabular}[c]{@{}c@{}}time for\\      solutions (s)\end{tabular} \\
                        \midrule
                        5                     & 205.95    & 205.95  & 0.00\%  & 0.15                                                                & 205.95    & 205.95  & 0.00\%  & 0.21                                                                \\
                        10                    & 463.37    & 441.83  & 4.65\%  & 0.68                                                                & 451.49    & 451.49  & 0.00\%  & 2.23                                                                \\
                        15                    & 694.64    & 639.78  & 7.90\%  & 4.92                                                                & 682.21    & 682.21  & 0.00\%  & 261.41                                                              \\
                        20                    & 933.08    & 844.53  & 9.49\%  & 82.20                                                               & 928.55    & 806.36  & 13.16\% & 3666.84                                                             \\
                        25                    & 1216.52   & 1101.26 & 9.48\%  & 3601.01                                                             & 1218.28   & 1002.56 & 17.71\% & 3618.68                                                             \\
                        30                    & 1451.12   & 1137.36 & 21.62\% & 3600.40                                                             & 1440.85   & 1071.36 & 25.64\% & 3601.36          \\
                        \bottomrule
                    \end{tabular}
                \end{table}

                We can see that the first approach cannot give optimal solutions even in small instances ($n\leq15$), but can significantly reduce the computation time and provide good bounds.
                In large instances ($n\geq15$), the first approach proves a smaller optimality gap than that of the relaxed Lagrangian methods and that of \texttt{MibS}.
                The second approach aims to find optimal solutions.
                As compared to the relaxed Lagrangian methods or \texttt{MibS}, the second approach obtains optimal solutions within shorter time in small instances and proves a much smaller optimality gap in large instances.
                The above results and analysis validate the effectiveness of trained decision rule-based valid inequalities.

    \section{Conclusions}\label{sec:conclusions}
        We consider bilevel mixed-integer linear programs with binary tender and develop exact algorithms based on valid inequalities.
        We first propose a family of Lagrangian-based valid inequalities as the baseline method to achieve global optimum and discuss the calculation of the Lagrangian coefficients.
        To further enhance the computational effectiveness, we then investigate valid inequalities based on submodularity and supermodularity and extend to cases of quasi-submodularity and quasi-supermodularity. In all these special cases, we derive more efficient calculation of the cut coefficients for the Lagrangian-based valid inequalities.
        Furthermore, we explore linear decision rule-based valid inequalities as a general enhancement method and exploit past solves of the lower-level formulation to obtain nonlinear decision rule-based valid inequalities through neural networks.
        Finally, we conduct extensive numerical experiments in general bilevel MILP instances and facility location interdiction problems, with the state-of-the-art solver \texttt{MibS} as the benchmark.
        The results demonstrate that the Lagrangian-based methods outperform \texttt{MibS} in both general problems and facility location interdiction problems.
        When (quasi-)submodularity exists, the \text{(quasi-)submodularity}-based methods can speed up branch-and-cut even more.
        Even when such special properties are absent, both linear decision rule-based and trained (nonlinear) decision rule-based methods can significantly improve the lower bound and reduce the optimality gap significantly in large-scale instances.

    \section*{Declarations}
        \paragraph{Funding and/or Conflicts of interest/Competing interests} Ruiwei Jiang is supported in part by the U.S. Air Force Office of Scientific Research under grant No. FA9550-23-1-0323. The authors declare no competing interests.

	\bibliographystyle{unsrt}
	\bibliography{reference}

    \clearpage
    \begin{appendices}
        \section{Baseline Algorithm}
            \subsection{Proof of Lemma \ref{lem:LR}}\label{apd:lem:LR}
                \begin{proof}\color{black}
                    We consdier two cases according to the feasibility of $x$.

                    Case 1: If $x\notin X_{LF}$,
                        we have $\phi(x)=-\infty$ according to \eqref{eq:value function}, and hence, $L_{\ell}(x, \lambda)\geq-\infty$ due to the weak duality.
                        We further consider two cases.\\
                        Case 1.1: If $L_{\ell}(x, \lambda)=-\infty$,
                            which means $\phi(z)$ is infeasible for all $z\in[0,1]^{n_{x}}$, we have
                            \begin{equation*}
                                \min_{\lambda\in\mathbb{R}^{n_{x}}}L_{\ell}(x,\lambda)=-\infty=\phi(x).
                            \end{equation*}
                        Case 1.2: If $L_{\ell}(x, \lambda)>-\infty$,
                            supposing that $z^*$ is the optimal solution to the right-hand solution of \eqref{eq:LR}, we have
                            \begin{equation*}
                                L_{\ell}(x,\lambda)=\phi(z^*)-\sum_{i\in[n_{x}]}\lambda_{i}\left(x_{i}-z_{i}^*\right).
                            \end{equation*}
                            Considering that $x$ is infeasible to $\phi$ and thus infeasible to $\psi$ while $z^*$ is feasible to $\psi$, $z^*\neq x$ must hold.
                            Hence, there exists $i\in[n_{x}]$ such that $x_{i}-z_{i}^*\neq0$.
                            If $x_{i}=1$, we have $x_{i}-z_{i}^*>0$.
                            When $\lambda_{i}\rightarrow+\infty$, we have $\lambda_{i}(x_{i}-z_{i}^*)\rightarrow+\infty$.
                            If $x_{i}=0$, we have $x_{i}-z_{i}^*<0$.
                            When $\lambda_{i}\rightarrow-\infty$, we have $\lambda_{i}(x_{i}-z_{i}^*)\rightarrow+\infty$.
                            With the bounded $\phi(z^*)$, we have $L_{\ell}(x,\lambda)\rightarrow-\infty$.
                            Hence, we have
                            \begin{equation*}
                                \min_{\lambda\in\mathbb{R}^{n_{x}}}L_{\ell}(x,\lambda)=-\infty=\phi(x).
                            \end{equation*}

                    Case 2: If $x\in X_{LF}$,
                        we have $\phi(x)>-\infty$ according to \eqref{eq:value function}, and hence, $L_{\ell}(x, \lambda)>-\infty$ due to the weak duality.
                        Supposing that $z^{*}(x,\lambda)$ is the optimal value to the right-hand side of \eqref{eq:LR}, we have
                        \begin{equation*}
                            L_{\ell}(x, \lambda)=\phi\big(z^{*}(x,\lambda)\big)-\sum_{i\in[n_{x}]}\lambda_{i}\big(x_{i}-z_{i}^{*}(x,\lambda)\big),
                        \end{equation*}
                        where for all $i\in[n_{x}]$, $x_{i}-z_{i}^{*}(x,\lambda)\geq0$ if $x_{i}=1$ and $x_{i}-z_{i}^{*}(x,\lambda)\leq0$ if $x_{i}=0$.\\
                        First, we prove that there exists $\lambda^{*}\in\mathbb{R}^{n_{x}}$ such that for all $\lambda$ that satisfies \eqref{eq:LR-strong-condition}, we have $x-z^{*}(x,\lambda)=0$.
                        Suppose the contrary that for all $\lambda^{*}\in\mathbb{R}^{n_{x}}$, there exists $\lambda$ that satisfies \eqref{eq:LR-strong-condition}, there exists $i\in[n_{x}]$ such that $x_{i}-z_{i}^{*}(x,\lambda)\neq0$.
                        If $x_{i}=1$, we have $\lambda_{i}\geq\lambda_{i}^{*}$ and $x_{i}-z_{i}^{*}(x,\lambda)>0$.
                        When $\lambda_{i}^{*}\rightarrow+\infty$, we have $\lambda_{i}=+\infty$ and then $\lambda_{i}\big(x_{i}-z_{i}^{*}(x,\lambda)\big)=+\infty$.
                        If $x_{i}=0$, we have $\lambda_{i}\leq\lambda_{i}^{*}$ and $x_{i}-z_{i}^{*}(x,\lambda)<0$.
                        When $\lambda_{i}^{*}\rightarrow-\infty$, we have $\lambda_{i}=-\infty$ and then $\lambda_{i}\big(x_{i}-z_{i}^{*}(x,\lambda)\big)=+\infty$.
                        Hence, $L_{\ell}(x, \lambda)=-\infty$, which contradicts with $L_{\ell}(x, \lambda)>-\infty$, and we complete this part of the proof.
                        Considering that the given $x$ may influence the value of $\lambda^{*}$, we use $\lambda^{*}(x)$ in the following for clarity.\\
                        Second, because for all $\lambda$ that satisfies \eqref{eq:LR-strong-condition}, $x-z^{*}(x,\lambda)=0$ holds, and hence, we have $L_{\ell}(x,\lambda)=\phi(x)$.
                        Furthermore, we have
                        \begin{equation*}
                            \min_{\lambda\in\mathbb{R}^{n_{x}}}L_{\ell}(x,\lambda)\leq\min_{\eqref{eq:LR-strong-condition}}L_{\ell}(x,\lambda)=\phi(x).
                        \end{equation*}
                        Combining the weak duality that $\phi(x)\leq\min_{\lambda\in\mathbb{R}^{n_{x}}}L_{\ell}(x,\lambda)$, we have $\phi(x)=\min_{\lambda\in\mathbb{R}^{n_{x}}}L_{\ell}(x,\lambda)$.
                        This completes the overall proof.
                \end{proof}

            \subsection{Proof of Corollary \ref{crl:LR}}\label{apd:crl:LR}
                \begin{proof}\color{black}
                    According to Lemma \ref{lem:LR}, for any $x\in\{0,1\}^{n_{x}}\cap X_{LF}$, there exists $\lambda^*(x)$ such that for all $\lambda$ that satisfies \eqref{eq:LR-strong-condition}, we have $\phi(x)=L_{\ell}(x,\lambda)$.
                    For all $i\in[n_{x}]$, we define $U_{i}:=\max_{x\in\{0,1\}^{n_{x}}\cap X_{LF}}\lambda_{i}^{*}(x)$ and $L_{i}:=\min_{x\in\{0,1\}^{n_{x}}\cap X_{LF}}\lambda_{i}^{*}(x)$.
                    Then, for any $x\in\{0,1\}^{n_{x}}\cap X_{LF}$ and any $\lambda$ that satisfies \eqref{eq:LR-strong-condition2}, it also holds that $\lambda$ satisfies \eqref{eq:LR-strong-condition} and thus  $\phi(x)=L_{\ell}(x,\lambda)$.
                \end{proof}

            \subsection{Proof of Proposition \ref{pps:LR-cut}}\label{apd:pps:LR-cut}
                \begin{proof}\color{black}
                    Because $\hat{\lambda}(1-x)$ satisfies \eqref{eq:LR-strong-condition2}, according to Corollary \ref{crl:LR}, for any $x\in[0,1]^{n_{x}}\cap X_{LF}$, we have $\phi(x)=L_{\ell}(x,\hat{\lambda}(1-x))$.
                    In the VFR \eqref{eq:VFR}, the lower-level feasibility $X_{LF}$ has been satisfied by \eqref{eq:lower feasibility}, we can directly replace $\phi(x)$ in \eqref{eq:optimality} by $L_{\ell}(x,\hat{\lambda}(1-x))$ as
                    \begin{equation*}
                        d_{\ell}^{\top}y\geq L_{\ell}(x,\hat{\lambda}(1-x))=\max_{z\in[0,1]^{n_{x}}}\left\{\psi(z)-(\hat{\lambda}(1-x))^{\top}(x-z)\right\},
                    \end{equation*}
                    which implies
                    \begin{equation*}
                        d_{\ell}^{\top}y\geq\psi(z)-(\hat{\lambda}(1-x))^{\top}(x-z),~\forall z\in[0,1]^{n_{x}}.
                    \end{equation*}
                    Hence, for any $z\in[0,1]^{n_{x}}$, \eqref{eq:LR-cut} is valid for \eqref{eq:optimality}.

                    For any $z\in\{0,1\}^{n_{x}}\subset[0,1]^{n_{x}}$, because $z$ is binary-valued, \eqref{eq:LR-cut} can be reformulated as
                    \begin{equation*}
                        \begin{aligned}
                            d_{\ell}^{\top}y&\geq\phi(z)-(\hat{\lambda}(1-x))^{\top}(x-z)\\
                            &=\phi(z)-\sum_{i\in[n_{x}]}[U_{i}x_{i}+L_{i}(1-x_{i})](x_{i}-z_{i})\\
                            &=\phi(z)-\sum_{i\in[n_{x}]}[U_{i}(1-z_{i})+L_{i}z_{i}](x_{i}-z_{i})\\
                            &=\phi(z)-(\hat{\lambda}(z))^{\top}(x-z).
                        \end{aligned}
                    \end{equation*}
                    By fixing $x$ on the right-hand side as $z$, we obtain
                    \begin{equation*}
                        d_{\ell}^{\top}y\geq \phi(z),
                    \end{equation*}
                    which implies that \eqref{eq:LR-cut2} is valid and tight for \eqref{eq:optimality}.
                    This completes the proof.
                \end{proof}

            \subsection{Proof of Lemma \ref{lem:ALR}}\label{apd:lem:ALR}
                \begin{proof}
                    We consider two cases according to the feasibility of $x$.

                    Case 1: If $x\notin X_{LF}$,
                        we have $\phi(x)=-\infty$ according to \eqref{eq:value function}, and hence, $L_{a}(x,\lambda, \rho)\geq-\infty$ due to the weak duality.
                        We further consider two cases.\\
                        Case 1.1: If $L_{a}(x,\lambda, \rho)=-\infty$,
                            which means $\phi(z)$ is infeasible for all $z\in[0,1]^{n_{x}}$, we have
                            \begin{equation*}
                                \min_{\rho\in\mathbb{R}_{+}}L_{a}(x,\lambda,\rho)=-\infty=\phi(x).
                            \end{equation*}
                        Case 1.2: If $L_{a}(x,\lambda, \rho)>-\infty$,
                            supposing that $z^*$ is the optimal solution to the right-hand solution of \eqref{eq:ALR}, we have
                            \begin{equation*}
                                L_{a}(x,\lambda,\rho)=\phi(z^*)-\sum_{i\in[n_{x}]}\lambda_{i}\left(x_{i}-z_{i}^*\right)-\rho\left(1^{\top}x+1^{\top}z^*-2x^{\top}z^*\right).
                            \end{equation*}
                            Considering that $x$ is infeasible to $\phi$ while $z^*$ is feasible to $\phi$, $z^*\neq x$ must hold, and hence, $1^{\top}x+1^{\top}z^*-2x^{\top}z^*\geq\|x-z^*\|^2_2>0$ and there exists $i\in[n_{x}]$ such that $x_{i}-z_{i}^*\neq0$.
                            When $\rho\rightarrow+\infty$, we have $\rho(1^{\top}x+1^{\top}z^*-2x^{\top}z^*)\rightarrow+\infty$.
                            If $x_{i}=1$, we have $x_{i}-z_{i}^*>0$.
                            When $\lambda_{i}\rightarrow+\infty$, we have $\lambda_{i}(x_{i}-z_{i}^*)\rightarrow+\infty$.
                            If $x_{i}=0$, we have $x_{i}-z_{i}^*<0$.
                            When $\lambda_{i}\rightarrow-\infty$, we have $\lambda_{i}(x_{i}-z_{i}^*)\rightarrow+\infty$.
                            With the bounded $\phi(z^*)$, we have $L_{a}(x,\lambda,\rho)\rightarrow-\infty$.
                            Hence, we have
                            \begin{equation*}
                                \min_{\rho\in\mathbb{R}_{+}}L_{a}(x,\lambda,\rho)=-\infty=\phi(x).
                            \end{equation*}

                    Case 2: If $x\in X_{LF}$,
                        we have $\phi(x)>-\infty$ according to \eqref{eq:value function}, and hence, $L_{a}(x, \lambda,\rho)>-\infty$ due to the weak duality.
                        Supposing that $z^{*}(x,\lambda,\rho)$ is the optimal value to the right-hand side of \eqref{eq:ALR}, we have
                        \begin{equation*}
                            L_{a}(x, \lambda,\rho)=\phi\big(z^{*}(x,\lambda,\rho)\big)-\sum_{i\in[n_{x}]}\lambda_{i}\big(x_{i}-z_{i}^{*}(x,\lambda,\rho)\big)-\rho\left(1^{\top}x+1^{\top}z^*(x,\lambda,\rho)-2x^{\top}z^*(x,\lambda,\rho)\right),
                        \end{equation*}
                        where
                        \begin{equation*}
                            1^{\top}x+1^{\top}z^*(x,\lambda,\rho)-2x^{\top}z^*(x,\lambda,\rho)\geq\|x-z^*(x,\lambda,\rho)\|^2_2\geq0
                        \end{equation*}
                        and for all $i\in[n_{x}]$, $x_{i}-z_{i}^{*}(x,\lambda,\rho)\geq0$ if $x_{i}=1$ and $x_{i}-z_{i}^{*}(x,\lambda,\rho)\leq0$ if $x_{i}=0$.\\
                        First, we prove that there exists $\lambda^{*}\in\mathbb{R}^{n_{x}}$ and $\rho^{*}\in\mathbb{R}_{+}$ such that for all $\lambda$ that satisfies \eqref{eq:ALR-strong-condition} and all $\rho\geq\rho^{*}$, we have $x-z^{*}(x,\lambda,\rho)=0$.
                        Suppose the contrary that for all $\lambda^{*}\in\mathbb{R}^{n_{x}}$ and $\rho^{*}\in\mathbb{R}_{+}$, there exists $\lambda$ that satisfies \eqref{eq:LR-strong-condition} and $\rho\geq\rho^{*}$, there exists $i\in[n_{x}]$ such that $x_{i}-z_{i}^{*}(x,\lambda,\rho)\neq0$.
                        Because $x_{i}-z_{i}^{*}(x,\lambda,\rho)\neq0$, we have $1^{\top}x+1^{\top}z^*(x,\lambda,\rho)-2x^{\top}z^*(x,\lambda,\rho)>0$.
                        When $\rho^{*}\rightarrow+\infty$, we have $\rho=+\infty$ and then $\rho\big(1^{\top}x+1^{\top}z^*(x,\lambda,\rho)-2x^{\top}z^*(x,\lambda,\rho)\big)=+\infty$.
                        If $x_{i}=1$, we have $\lambda_{i}\geq\lambda_{i}^{*}$ and $x_{i}-z_{i}^{*}(x,\lambda,\rho)>0$.
                        When $\lambda_{i}^{*}\rightarrow+\infty$, we have $\lambda_{i}=+\infty$ and then $\lambda_{i}\big(x_{i}-z_{i}^{*}(x,\lambda,\rho)\big)=+\infty$.
                        If $x_{i}=0$, we have $\lambda_{i}\leq\lambda_{i}^{*}$ and $x_{i}-z_{i}^{*}(x,\lambda,\rho)<0$.
                        When $\lambda_{i}^{*}\rightarrow-\infty$, we have $\lambda_{i}=-\infty$ and then $\lambda_{i}\big(x_{i}-z_{i}^{*}(x,\lambda,\rho)\big)=+\infty$.
                        Hence, $L_{a}(x, \lambda,\rho)=-\infty$ which contradicts with $L_{a}(x, \lambda,\rho)>-\infty$ and we complete this part of the proof.
                        Considering that the given $x$ may influence the value of $\lambda^{*}$ and $\rho^*$, we use $\lambda^{*}(x)$ and $\rho^{*}(x)$ in the following for clarity.\\
                        Second, because for all $\lambda$ that satisfies \eqref{eq:ALR-strong-condition} and all $\rho\geq\rho^{*}$, $x-z^{*}(x,\lambda,\rho)=0$ holds, and hence, we have $L_{a}(x,\lambda,\rho)=\phi(x)$.
                        Furthermore, we have
                        \begin{equation*}
                            \min_{\lambda\in\mathbb{R}^{n_{x}},\rho\in\mathbb{R}_{+}}L_{a}(x,\lambda,\rho)\leq\min_{\eqref{eq:ALR-strong-condition},\rho\geq\rho^{*}(x)}L_{a}(x,\lambda,\rho)=\phi(x).
                        \end{equation*}
                        Combining the weak duality, we have $\phi(x)=\min_{\lambda\in\mathbb{R}^{n_{x}},\rho\in\mathbb{R}_{+}}L_{a}(x,\lambda,\rho)$.
                        This completes the overall proof.
                \end{proof}

            \subsection{Proof of Corollary \ref{crl:ALR}}\label{apd:crl:ALR}
                \begin{proof}
                    According to Lemma \ref{lem:ALR}, for any $x\in\{0,1\}^{n_{x}}\cap X_{LF}$, there exists $\rho^*(x)$ and $\lambda^*(x)$ such that for all $\rho\geq\rho^*(x)$ and for all $\lambda$ that satisfies \eqref{eq:ALR-strong-condition}, we have $\phi(x)=L_{a}(x,\lambda,\rho)$.
                    We define $\hat{\rho}:=\max_{x\in\{0,1\}^{n_{x}}\cap X_{LF}}\rho^{*}(x)$.
                    For all $i\in[n_{x}]$, we define $U_{i}:=\max_{x\in\{0,1\}^{n_{x}}\cap X_{LF}}\lambda_{i}^{*}(x)$ and $L_{i}:=\min_{x\in\{0,1\}^{n_{x}}\cap X_{LF}}\lambda_{i}^{*}(x)$.
                    Then, for any $x\in\{0,1\}^{n_{x}}\cap X_{LF}$, any $\rho\geq\hat{\rho}$, and any $\lambda$ that satisfies \eqref{eq:ALR-strong-condition2}, it also holds that $\rho\geq\rho^*(x)$ and $\lambda$ satisfies \eqref{eq:ALR-strong-condition} and thus $\phi(x)=L_{a}(x,\lambda,\rho)$.
                \end{proof}

            \subsection{Proof of Proposition \ref{pps:ALR-cut}}\label{apd:pps:ALR-cut}
                \begin{proof}
                    Because $\hat{\lambda}(1-x)$ satisfies \eqref{eq:ALR-strong-condition2}, according to Corollary \ref{crl:ALR}, for any $x\in[0,1]^{n_{x}}\cap X_{LF}$, we have $\phi(x)=L_{a}(x,\hat{\lambda}(1-x),\hat{\rho})$.
                    In the VFR \eqref{eq:VFR}, the lower-level feasibility $X_{LF}$ has been satisfied by \eqref{eq:lower feasibility}, we can directly replace $\phi(x)$ in \eqref{eq:optimality} by $L_{a}(x,\hat{\lambda}(1-x),\hat{\rho})$ as
                    \begin{equation*}
                        d_{\ell}^{\top}y\geq L_{a}(x,\hat{\lambda}(1-x),\hat{\rho})=\max_{z\in[0,1]^{n_{x}}}\left\{\phi(z)-(\hat{\lambda}(1-x))^{\top}(x-z)-\hat{\rho}\left(1^{\top}x+1^{\top}z-2x^{\top}z\right)\right\},
                    \end{equation*}
                    which implies
                    \begin{equation*}
                        d_{\ell}^{\top}y\geq\phi(z)-(\hat{\lambda}(1-x))^{\top}(x-z)-\hat{\rho}\left(1^{\top}x+1^{\top}z-2x^{\top}z\right),~\forall z\in[0,1]^{n_{x}}.
                    \end{equation*}
                    Hence, for any $z\in[0,1]^{n_{x}}$, \eqref{eq:ALR-cut} is valid for \eqref{eq:optimality}.

                    For any $z\in\{0,1\}^{n_{x}}\subset[0,1]^{n_{x}}$, because $z$ is binary-valued, \eqref{eq:ALR-cut} can be reformulated as
                    \begin{equation*}
                        \begin{aligned}
                            d_{\ell}^{\top}y&\geq\phi(z)-(\hat{\lambda}(1-x))^{\top}(x-z)-\hat{\rho}\left(1^{\top}x+1^{\top}z-2x^{\top}z\right)\\
                            &=\phi(z)-\sum_{i\in[n_{x}]}[U_{i}x_{i}+L_{i}(1-x_{i})](x_{i}-z_{i})-\hat{\rho}\left(1^{\top}x+1^{\top}z-2x^{\top}z\right)\\
                            &=\phi(z)-\sum_{i\in[n_{x}]}[U_{i}(1-z_{i})+L_{i}z_{i}](x_{i}-z_{i})-\hat{\rho}\left(1^{\top}x+1^{\top}z-2x^{\top}z\right)\\
                            &=\phi(z)-(\hat{\lambda}(z))^{\top}(x-z)-\hat{\rho}\left(1^{\top}x+1^{\top}z-2x^{\top}z\right).
                        \end{aligned}
                    \end{equation*}
                    By fixing $x$ on the right-hand side as $z$, we obtain
                    \begin{equation*}
                        d_{\ell}^{\top}y\geq \phi(z),
                    \end{equation*}
                    which implies that \eqref{eq:ALR-cut2} is valid and tight for \eqref{eq:optimality}.
                    This completes the proof.
                \end{proof}

            \subsection{Proof of Corollary \ref{crl:representation}}\label{apd:crl:representation}
                \begin{proof}
                    We use the penalty-based valid inequality \eqref{eq:penalty-cut2} for instance and the case using valid inequalities \eqref{eq:LR-cut2} or \eqref{eq:ALR-cut2} can be similarly proved.

                    In the VFR \eqref{eq:VFR}, $x\in X_{LF}\subseteq\{0,1\}^{n_{x}}$ holds.
                    We first prove $\mathcal{F}_{1}=\mathcal{F}_{2}$, where
                    \begin{align*}
                        &\mathcal{F}_{1}:=\left\{
                            (x,y)\in X_{LF}\times Y:d_{\ell}^{\top}y\geq\phi(x)
                        \right\},\\
                        &\mathcal{F}_{2}:=\left\{
                            (x,y)\in X_{LF}\times Y:d_{\ell}^{\top}y\geq\phi(z)-\hat{\rho}\left(1^{\top}x+1^{\top}z-2x^{\top}z\right),~\forall z\in X_{LF}
                        \right\}.
                    \end{align*}
                    First, from Proposition \ref{pps:penalty-cut}, for any $z\in X_{LF}$, \eqref{eq:penalty-cut2} is valid for \eqref{eq:optimality}.
                    Hence, for any $(x,y)\in\mathcal{F}_{1}$, $(x,y)\in\mathcal{F}_{2}$ must holds.
                    Second, suppose an arbitrary $(\hat{x},\hat{y})\notin\mathcal{F}_{1}$, that is, $d_{\ell}^{\top}\hat{y}<\phi(\hat{x})$.
                    Obviously, $(\hat{x},\hat{y})$ violates the inequality \eqref{eq:penalty-cut2} when $z=\hat{x}$.
                    It means that for any $(x,y)\notin\mathcal{F}_{1}$, $(x,y)\notin\mathcal{F}_{2}$ must holds.
                    Therefore, $\mathcal{F}_{1}=\mathcal{F}_{2}$ and we complete this part of proof.

                    Because $X_{LF}\subseteq\{0,1\}^{n_{x}}$, we have $|X_{LF}|\leq2^{n_{x}}$ is finite.
                    Therefore, \eqref{eq:optimality} can be fully described by a finite number of inequalities \eqref{eq:penalty-cut2}.
                \end{proof}

            \subsection{Proof of Proposition \ref{pps:selection}}\label{apd:pps:selection}
                \begin{proof}\color{black}
                    According to Lemma \ref{lem:solution}, we have
                    \begin{equation*}
                        L(x, \lambda,\rho)=\max_{z\in\{0,1\}^{n_{x}}}~\bar{L}(x,\lambda,\rho,z).
                    \end{equation*}
                    To guarantee $\phi(x)=L(x, \lambda,\rho)$ for all $x\in\{0,1\}^{n_{x}}$, we require for all $x\in\{0,1\}^{n_{x}}$,
                    \begin{equation*}
                        x\in\arg\max_{z\in\{0,1\}^{n_{x}}}\bar{L}(x,\lambda,\rho,z).
                    \end{equation*}
                    It can be imposed by for all $x\in\{0,1\}^{n_{x}}$, for all $i\in[n_{x}]$,
                    \begin{equation}\label{eq:impose}
                        \bar{L}(x,\lambda,\rho,z_{-i},z_{i}=x_{i})\geq\bar{L}(x,\lambda,\rho,z_{-i},z_{i}=1-x_{i}), ~\forall z_{-i}\in\{0,1\}^{n_{x}-1}.
                    \end{equation}
                    Then, we analyze the condition for different valid inequalities.

                    (1) For the penalty-based relaxation \eqref{eq:penalty}, \eqref{eq:impose} implies for all $x\in\{0,1\}^{n_{x}}$, for all $i\in[n_{x}]$,
                    \begin{equation*}
                        \rho\geq\phi(z_{-i},z_{i}=1-x_{i})-\phi(z_{-i},z_{i}=x_{i}),~\forall z_{-i}\in\{0,1\}^{n_{x}-1},
                    \end{equation*}
                    which further implies
                    \begin{equation*}
                        \rho\geq\max_{\substack{z,z'\in\{0,1\}^{n_{x}}\\\|z-z'\|_{2}^{2}=1}}\phi(z)-\phi(z').
                    \end{equation*}
                    Hence, $\hat{\rho}$ can be selected as \eqref{eq:selection-penalty}

                    (2) For the Lagrangian-based relaxation \eqref{eq:LR}, \eqref{eq:impose} implies for all $x\in\{0,1\}^{n_{x}}$, for all $i\in[n_{x}]$,
                    \begin{equation*}
                        \lambda_{i}(2x_{i}-1)\geq\phi(z_{-i},z_{i}=1-x_{i})-\phi(z_{-i},z_{i}=x_{i}),~\forall z_{-i}\in\{0,1\}^{n_{x}-1}.
                    \end{equation*}
                    Considering that $x_{i}\in\{0,1\}$, we have for all $x\in\{0,1\}^{n_{x}}$, for all $i\in[n_{x}]$,
                    \begin{equation*}
                        \left\{\begin{aligned}
                            x_{i}=0\Rightarrow\lambda_{i}\leq\phi(z_{-i},z_{i}=0)-\phi(z_{-i},z_{i}=1),~\forall z_{-i}\in\{0,1\}^{n_{x}-1}\\
                            x_{i}=1\Rightarrow\lambda_{i}\geq\phi(z_{-i},z_{i}=0)-\phi(z_{-i},z_{i}=1),~\forall z_{-i}\in\{0,1\}^{n_{x}-1}
                        \end{aligned}\right.,
                    \end{equation*}
                    which further implies for all $x\in\{0,1\}^{n_{x}}$, for all $i\in[n_{x}]$,
                    \begin{equation*}
                        \left\{\begin{aligned}
                            x_{i}=0\Rightarrow\lambda_{i}\leq\min_{\substack{z,z'\in\{0,1\}^{n_{x}}\\z_{-i}=z'_{-i},z_{i}=0,z'_{i}=1}}\phi(z)-\phi(z')\\
                            x_{i}=1\Rightarrow\lambda_{i}\geq\max_{\substack{z,z'\in\{0,1\}^{n_{x}}\\z_{-i}=z'_{-i},z_{i}=0,z'_{i}=1}}\phi(z)-\phi(z')
                        \end{aligned}\right.
                    \end{equation*}
                    Hence, $U_{i}$ and $L_{i}$ can be selected as \eqref{eq:selection-LR}.

                    (3) For the augmented Lagrangian-based relaxation \eqref{eq:ALR}, \eqref{eq:impose} implies for all $x\in\{0,1\}^{n_{x}}$, for all $i\in[n_{x}]$,
                    \begin{equation*}
                        \lambda_{i}(2x_{i}-1)+\rho\geq\phi(z_{-i},z_{i}=1-x_{i})-\phi(z_{-i},z_{i}=x_{i}),~\forall z_{-i}\in\{0,1\}^{n_{x}-1}.
                    \end{equation*}
                    Considering that $x_{i}\in\{0,1\}$, we have for all $x\in\{0,1\}^{n_{x}}$, for all $i\in[n_{x}]$,
                    \begin{equation*}
                        \left\{\begin{aligned}
                            x_{i}=0\Rightarrow\lambda_{i}-\rho\leq\phi(z_{-i},z_{i}=0)-\phi(z_{-i},z_{i}=1),~\forall z_{-i}\in\{0,1\}^{n_{x}-1}\\
                            x_{i}=1\Rightarrow\lambda_{i}+\rho\geq\phi(z_{-i},z_{i}=0)-\phi(z_{-i},z_{i}=1),~\forall z_{-i}\in\{0,1\}^{n_{x}-1}
                        \end{aligned}\right.,
                    \end{equation*}
                    which further implies for all $x\in\{0,1\}^{n_{x}}$, for all $i\in[n_{x}]$,
                    \begin{equation*}
                        \left\{\begin{aligned}
                            x_{i}=0&\Rightarrow\lambda_{i}\leq\rho+\min_{\substack{z,z'\in\{0,1\}^{n_{x}}\\z_{-i}=z'_{-i},z_{i}=0,z'_{i}=1}}\phi(z)-\phi(z')\\
                            x_{i}=1&\Rightarrow\lambda_{i}\geq-\rho+\max_{\substack{z,z'\in\{0,1\}^{n_{x}}\\z_{-i}=z'_{-i},z_{i}=0,z'_{i}=1}}\phi(z)-\phi(z')
                        \end{aligned}\right.
                    \end{equation*}
                    Hence, $\hat{\rho}$, $U_{i}$, and $L_{i}$ can be selected as \eqref{eq:selection-ALR}
                \end{proof}

            \subsection{Proof of Proposition \ref{pps:selection-quick}}\label{apd:pps:selection-quick}
                \begin{proof}\color{black}
                    We combine the detailed formulation of $\phi$ and relax the calculation in Proposition \ref{pps:selection}.

                    (1) The right-hand side of \eqref{eq:selection-penalty} is relaxed as
                    \begin{equation*}
                        \begin{aligned}
                            &\max_{\substack{z,z'\in\{0,1\}^{n_{x}}\\\|z-z'\|_{2}^{2}=1}}\left\{
                                \max_{\substack{y\in Y\\B_{\ell}y\leq h_{\ell}-A_{\ell}z}}d_{\ell}^{\top}y-\max_{\substack{y'\in Y\\B_{\ell}y'\leq h_{\ell}-A_{\ell}z'}}d_{\ell}^{\top}y'
                            \right\}\\
                            \leq&\max_{\substack{z,z'\in\{0,1\}^{n_{x}}\\\|z-z'\|_{2}^{2}=1}}\left\{
                                \max_{\substack{y\in Y\\B_{\ell}y\leq h_{\ell}-A_{\ell}z}}d_{\ell}^{\top}y-\min_{\substack{y'\in Y\\B_{\ell}y'\leq h_{\ell}-A_{\ell}z'}}d_{\ell}^{\top}y'
                            \right\}
                            =\max_{\substack{z,z'\in\{0,1\}^{n_{x}}\\\|z-z'\|_{2}^{2}=1\\y,y'\in Y\\B_{\ell}y\leq h_{\ell}-A_{\ell}z\\B_{\ell}y'\leq h_{\ell}-A_{\ell}z'}}d_{\ell}^{\top}y-d_{\ell}^{\top}y'
                        \end{aligned}
                    \end{equation*}
                    Hence, $\hat{\rho}$ can be selected as \eqref{eq:selection-penalty-quick}.

                    (2) The right-hand side of the first equation in \eqref{eq:selection-LR} is relaxed as
                    \begin{equation*}
                        \begin{aligned}
                            &\min_{\substack{z,z'\in\{0,1\}^{n_{x}}\\z_{-i}=z'_{-i},z_{i}=0,z'_{i}=1}}\left\{
                                \max_{\substack{y\in Y\\B_{\ell}y\leq h_{\ell}-A_{\ell}z}}d_{\ell}^{\top}y-\max_{\substack{y'\in Y\\B_{\ell}y'\leq h_{\ell}-A_{\ell}z'}}d_{\ell}^{\top}y'
                            \right\}\\
                            \geq&\min_{\substack{z,z'\in\{0,1\}^{n_{x}}\\z_{-i}=z'_{-i},z_{i}=0,z'_{i}=1}}\left\{
                                \min_{\substack{y\in Y\\B_{\ell}y\leq h_{\ell}-A_{\ell}z}}d_{\ell}^{\top}y-\max_{\substack{y'\in Y\\B_{\ell}y'\leq h_{\ell}-A_{\ell}z'}}d_{\ell}^{\top}y'
                            \right\}
                            =\min_{\substack{z,z'\in\{0,1\}^{n_{x}}\\z_{-i}=z'_{-i},z_{i}=0,z'_{i}=1\\y,y'\in Y\\B_{\ell}y\leq h_{\ell}-A_{\ell}z\\B_{\ell}y'\leq h_{\ell}-A_{\ell}z'}}d_{\ell}^{\top}y-d_{\ell}^{\top}y'
                        \end{aligned}
                    \end{equation*}
                    The right-hand side of the second equation in \eqref{eq:selection-LR} is relaxed as
                    \begin{equation*}
                        \begin{aligned}
                            &\max_{\substack{z,z'\in\{0,1\}^{n_{x}}\\z_{-i}=z'_{-i},z_{i}=0,z'_{i}=1}}\left\{
                                \max_{\substack{y\in Y\\B_{\ell}y\leq h_{\ell}-A_{\ell}z}}d_{\ell}^{\top}y-\max_{\substack{y'\in Y\\B_{\ell}y'\leq h_{\ell}-A_{\ell}z'}}d_{\ell}^{\top}y'
                            \right\}\\
                            \leq&\max_{\substack{z,z'\in\{0,1\}^{n_{x}}\\z_{-i}=z'_{-i},z_{i}=0,z'_{i}=1}}\left\{
                                \max_{\substack{y\in Y\\B_{\ell}y\leq h_{\ell}-A_{\ell}z}}d_{\ell}^{\top}y-\min_{\substack{y'\in Y\\B_{\ell}y'\leq h_{\ell}-A_{\ell}z'}}d_{\ell}^{\top}y'
                            \right\}
                            =\max_{\substack{z,z'\in\{0,1\}^{n_{x}}\\z_{-i}=z'_{-i},z_{i}=0,z'_{i}=1\\y,y'\in Y\\B_{\ell}y\leq h_{\ell}-A_{\ell}z\\B_{\ell}y'\leq h_{\ell}-A_{\ell}z'}}d_{\ell}^{\top}y-d_{\ell}^{\top}y'
                        \end{aligned}
                    \end{equation*}
                    Hence, $U/L$ can be selected as \eqref{eq:selection-LR-quick}.
                \end{proof}

            \subsection{Proof of Proposition \ref{pps:selection-special}}\label{apd:pps:selection-special}
                \begin{proof}\color{black}
                    Comparing \eqref{eq:selection-penalty} and \eqref{eq:selection-LR}, we have
                    \begin{equation*}
                        \hat{\rho}=\max_{i\in[n_{x}]}\left\{\max\{-L_{i},U_{i}\}\right\}.
                    \end{equation*}
                    Then, we focus on the calculation of $L_{i}$ and $U_{i}$.

                    If $\phi(x)$ is submodular in $x\in\{0,1\}^{n_{x}}$, according to Definition \ref{def:property}, for any $x_{1}$, $x_{2}\in\{0,1\}^{n_{x}}$ with $x_{1}<x_{2}$, for any $i\in[n_{x}]$ with $x_{1}\wedge e_{i}=0$ and $x_{2}\wedge e_{i}=0$, we have
                    \begin{equation*}
                        \phi(x_{1})-\phi(x_{1}+e_{i})\leq\phi(x_{2})-\phi(x_{2}+e_{i}).
                    \end{equation*}
                    Hence, we simplify the calculation of \eqref{eq:selection-LR} by
                    \begin{equation*}
                        L_{i}=\min_{\substack{z,z'\in\{0,1\}^{n_{x}}\\z_{-i}=z'_{-i},z_{i}=0,z'_{i}=1}}\phi(z)-\phi(z')=\min_{\substack{z\in\{0,1\}^{n_{x}}\\z_{i}=0}}\phi(z)-\phi(z+e_{i})=\phi(0)-\phi(e_{i}),
                    \end{equation*}
                    \begin{equation*}
                        U_{i}=\max_{\substack{z,z'\in\{0,1\}^{n_{x}}\\z_{-i}=z'_{-i},z_{i}=0,z'_{i}=1}}\phi(z)-\phi(z')=\max_{\substack{z\in\{0,1\}^{n_{x}}\\z_{i}=0}}\phi(z)-\phi(z+e_{i})=\phi(1-e_{i})-\phi(1).
                    \end{equation*}

                    If $\phi(x)$ is supermodular in $x\in\{0,1\}^{n_{x}}$, according to Definition \ref{def:property}, for any $x_{1}$, $x_{2}\in\{0,1\}^{n_{x}}$ with $x_{1}<x_{2}$, for any $i\in[n_{x}]$ with $x_{1}\wedge e_{i}=0$ and $x_{2}\wedge e_{i}=0$, we have
                    \begin{equation*}
                        \phi(x_{1})-\phi(x_{1}+e_{i})\geq\phi(x_{2})-\phi(x_{2}+e_{i}).
                    \end{equation*}
                    Hence, we simplify the calculation of \eqref{eq:selection-LR} by
                    \begin{equation*}
                        L_{i}=\min_{\substack{z,z'\in\{0,1\}^{n_{x}}\\z_{-i}=z'_{-i},z_{i}=0,z'_{i}=1}}\phi(z)-\phi(z')=\min_{\substack{z\in\{0,1\}^{n_{x}}\\z_{i}=0}}\phi(z)-\phi(z+e_{i})=\phi(1-e_{i})-\phi(1),
                    \end{equation*}
                    \begin{equation*}
                        U_{i}=\max_{\substack{z,z'\in\{0,1\}^{n_{x}}\\z_{-i}=z'_{-i},z_{i}=0,z'_{i}=1}}\phi(z)-\phi(z')=\max_{\substack{z\in\{0,1\}^{n_{x}}\\z_{i}=0}}\phi(z)-\phi(z+e_{i})=\phi(0)-\phi(e_{i}).
                    \end{equation*}
                \end{proof}

        \vspace{-6ex}
        \section{Decision Rule-based Method}
        \vspace{-0ex}
            \subsection{Learning Framwork for Trained Decision Rule}\label{apd:learning}
                \begin{figure}[!b]
                    \begin{center}
                        \vspace{-2ex}
                        \includegraphics[width=0.70\columnwidth]{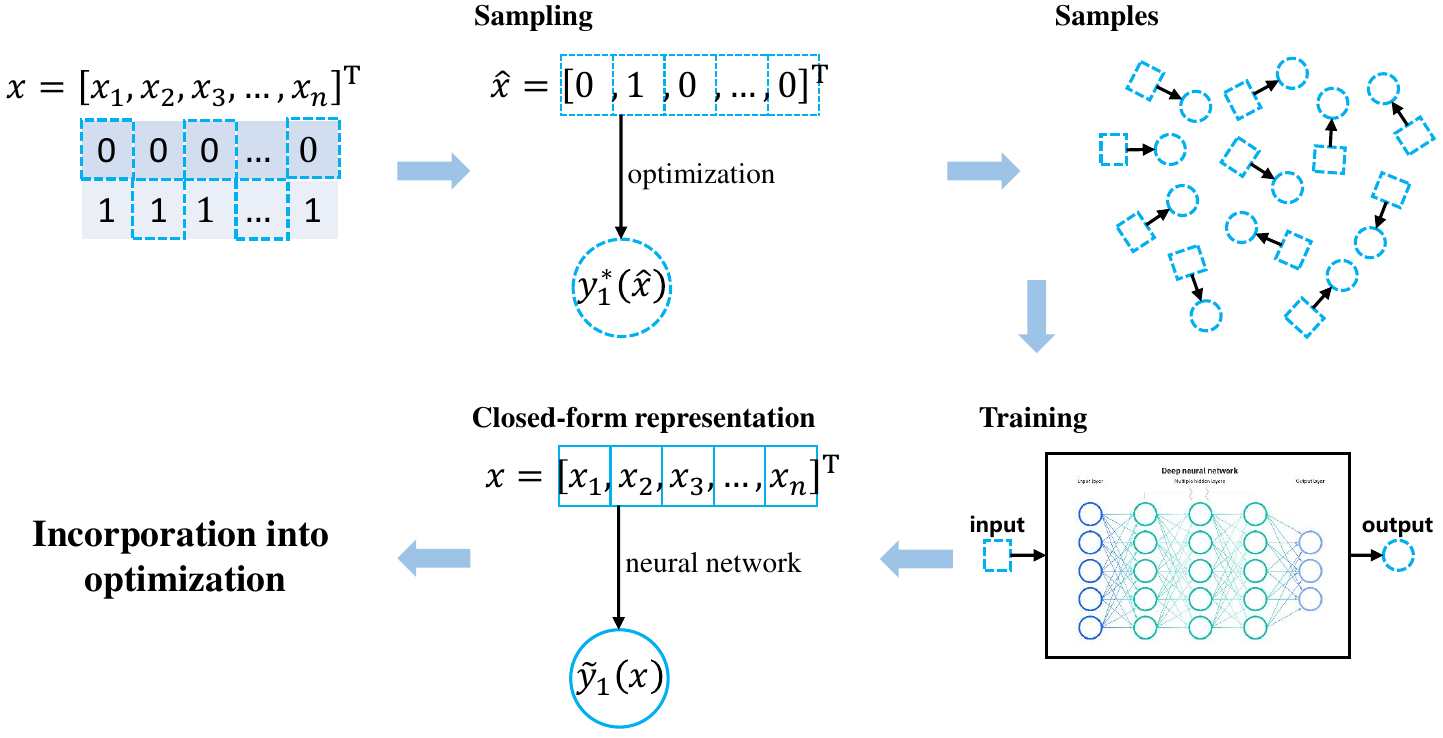}
                    \end{center}
                    \vspace{-2ex}
                    \caption{The overview of the process to get the trained decision rule in Section \ref{sec:DR}.}
                    \vspace{-1ex}
                    \label{fig:approximation}
                \end{figure}
                Figure \ref{fig:approximation} illustrates the adopted learning framework.
                We consider a set of the upper-level decision variables $\hat{x}$ and compute/record the lower-level optimal solution $y_{1}^*(\hat{x})$ to obtain sample-label pairs $(\hat{x},y_{1}^*(\hat{x}))$.
                After we have sampled or recorded sufficient sample-label pairs $(\hat{x},y_{1}^*(\hat{x}))$, we adopt supervised learning to train a neural network $\tilde{y}_{1}(x)$ to fit the mapping $y_{1}^*(x)$.
                More details can be found in our previous paper \cite{zhoulearning}.

            \subsection{Approximation of Sigmoid Activation}\label{apd:sigmoid}
                Figure \ref{fig:sigmoid} shows the comparison between $\sigma(x)$ and $\tilde{\sigma}_{p}(x)$.
                \begin{figure}[!htbp]
                    \begin{center}
                        \vspace{-0ex}
                        \includegraphics[width=0.5\columnwidth]{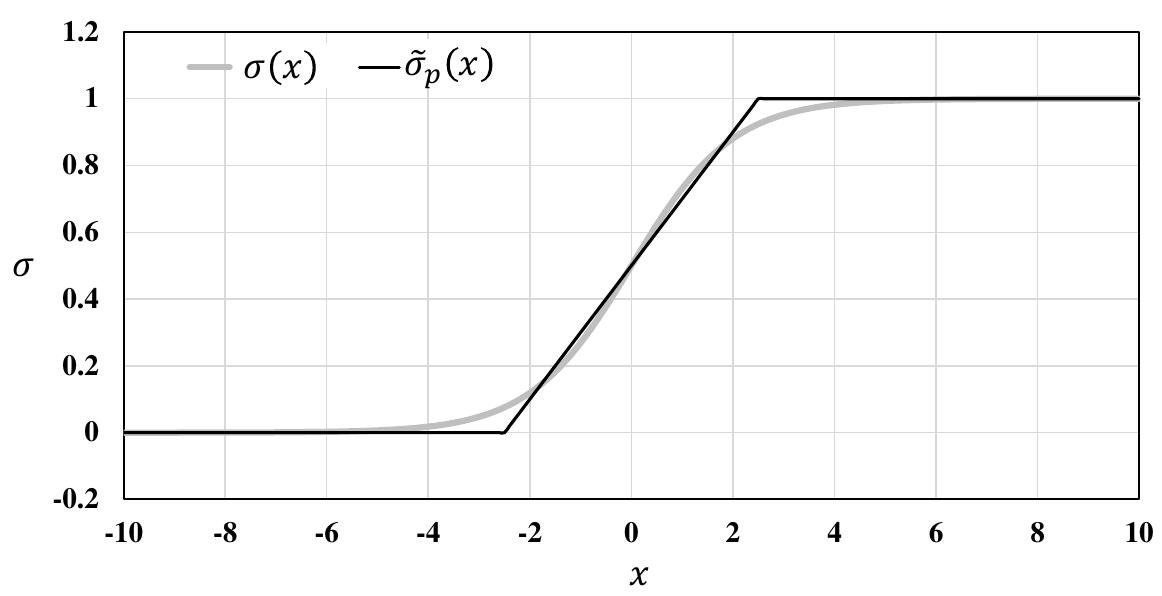}
                    \end{center}
                    \vspace{-2ex}
                    \caption{Comparison between $\sigma(x)$ and $\tilde{\sigma}_{p}(x)$.}
                    \vspace{-4ex}
                    \label{fig:sigmoid}
                \end{figure}

        \section{Instance Settings}
            \subsection{Instance Generation for General bilevel MILP}\label{apd:instance-general}
                We follow the instance generation rules of \cite{tahernejad2020branch} for general problems with the following formulatiuon:
                \begin{align*}
                    \min_{x,y}&~c_{u}^{\top}x+d_{u}^{\top}y\\
                    \textrm{s.t.}&~A_{u}x+B_{u}y\leq h_{u}\\
                    &~x\in\{0,1\}^{n_{x}}\\
                    &~\begin{aligned}
                            y\in\arg\max_{y}&~d_{\ell}^{\top}y\\
                            \textrm{s.t.}&~A_{\ell}x+B_{\ell}y\leq h_{\ell}\\
                            &~y_{I_{b}}\in\{0,1\}^{|I_{b}|}, y_{I_{c}}\in[0,1]^{|I_{c}|},
                        \end{aligned}
                \end{align*}
                where $I_{b}\subseteq[n_{y}]$ is the index set of binary variables in $y$ and $I_{c}=[n_{y}]\backslash I_{b}$ is the index set of continuous variables in $y$.
                In all generated instances, we set $n_{x}=n_{y}$ and set the number of constraints as $0.4n_{x}$ in the upper level and $0.4n_{y}$ in the lower level.
                We set the number of the binary variables in $y$ as $0.5n_{y}$, i.e., $|I_{b}|=0.5n_{y}$, and thus $|I_{c}|=n_{y}-|I_{b}|=0.5n_{y}$.
                The coefficients in $c_{u}$, $d_{u}$, and $d_{\ell}$ are uniformly distributed within [-50, 50].
                The coefficients in $A_{u}$, $B_{u}$, $A_{\ell}$, and $B_{\ell}$ are uniformly distributed within [0, 10].
                The coefficients in $h_{u}$ and $h_{\ell}$ are uniformly distributed within [30, 130] and [10, 110], respectively.

            \subsection{Instance Generation for Facility Location Interdiction Problems}\label{apd:instance-facility}
                For instances of facility location interdiction problems, we set $m=10n$, $B=\text{round}(n/3)$, and $C'=C\times B/n$.
                Each entry of demand $d$ is uniformly distributed within [0, 1].
                Each entry of original capacity $C$ is uniformly distributed within $[1/3, 1]\times m/n$ and satisfy the following three conditions to avoid trivial instances:
                i) the total demand is no more than the total original capacity,
                ii) without repair, the total demand may be more than the available capacity under interdiction,
                iii) with repair, the total demand is no more than the available capacity under worst-case interdiction.
                Each entry of price $c_{t}$ are uniformly distributed within [0, 1], $c_{p}=10$, and $c_{r,i}=2C'_{i}\max_{j\in[m]}\{c_{t,ij}\},\forall i\in[n]$.
    \end{appendices}

\end{document}